%% file: chromatiques_v6.tex
\documentclass[11pt,twoside,a4paper]{article}
\usepackage{color,amscd,amsmath,amssymb,amsthm,latexsym,stmaryrd,authblk,mathabx,shuffle,hyperref,tikz,tkz-tab,mathrsfs} 
\usepackage[latin1]{inputenc}
\usepackage[T1]{fontenc}   
\usepackage[english]{babel}
\providecommand{\AMSclass}[1]{\textbf{\textit{AMS classification.}} #1}

\input{xy}
\xyoption{all}

\definecolor{vert}{rgb}{0.,0.5,0.}

\title{Chromatic polynomials and bialgebras of graphs}
\date{}
\author{Lo\"\i c Foissy}
\affil{\small{Univ. Littoral Côte d'Opale, UR 2597
LMPA, Laboratoire de Mathématiques Pures et Appliquées Joseph Liouville
F-62100 Calais, France}.\\ Email: \texttt{foissy@univ-littoral.fr}}

\theoremstyle{plain}
\newtheorem{theo}{Theorem}[section]
\newtheorem{lemma}[theo]{Lemma}
\newtheorem{cor}[theo]{Corollary}
\newtheorem{prop}[theo]{Proposition}
\newtheorem{defi}[theo]{Definition}

\theoremstyle{remark}
\newtheorem{remark}{Remark}[section]
\newtheorem{notation}{Notations}[section]
\newtheorem{example}{Example}[section]

\newcommand{\N}{\mathbb{N}}

\newcommand{\Q}{\mathbb{Q}}

\renewcommand{\geq}{\geqslant}
\renewcommand{\leq}{\leqslant}

\newcommand{\id}{\mathrm{Id}}

\newcommand{\Sym}{\mathbf{Sym}}
\newcommand{\WSym}{\mathbf{WSym}}
\newcommand{\QSym}{\mathbf{QSym}}
\newcommand{\WQSym}{\mathbf{WQSym}}

\newcommand{\gr}{\mathcal{G}}
\newcommand{\GR}{\mathscr{G}}
\newcommand{\hgr}{\mathcal{H}_{\gr}}
\newcommand{\hgrd}{\mathcal{H}_{\gr(\calD)}}

\newcommand{\hGR}{\mathcal{H}_{\GR}}
\newcommand{\mgr}{M_{\gr}}
\newcommand{\mgrd}{M_{\gr(\calD)}}
\newcommand{\calf}{\mathcal{F}}
\newcommand{\F}{\mathbb{F}}
\newcommand{\SP}{\mathcal{SP}}
\newcommand{\PW}{\mathcal{PW}}
\newcommand{\VC}{\mathbb{VC}}
\newcommand{\PC}{\mathbb{PC}}
\newcommand{\PVC}{\mathbb{PVC}}
\newcommand{\IP}{\mathbb{IP}}
\newcommand{\rel}{\mathcal{R}}
\newcommand{\X}{\mathbf{X}}
\newcommand{\Y}{\mathbf{Y}}
\newcommand{\bfF}{\mathbf{F}}
\newcommand{\calF}{\mathcal{F}}
\newcommand{\calD}{\mathcal{D}}
\newcommand{\rep}{\mathrm{rep}}
\newcommand{\Rep}{\mathrm{Rep}}
\newcommand{\pack}{\mathrm{Pack}}
\newcommand{\wt}{\mathrm{wt}}
\newcommand{\calA}{\mathfrak{A}}

\input{graphes}

\begin{document}

\maketitle

\begin{abstract}
The chromatic polynomial is characterized as the unique polynomial invariant of graphs, compatible with two interacting bialgebras structures: 
the first coproduct is given by partitions of vertices into two parts, the second one by a contraction-extraction process.
This gives Hopf-algebraic proofs of Rota's result on the signs of coefficients of chromatic polynomials and of Stanley's interpretation
of the values at negative integers of chromatic polynomials. We also consider chromatic symmetric functions and their noncommutative versions.
\end{abstract}

\AMSclass{16T30 05C15 05C31}

\tableofcontents

\section*{Introduction}

In graph theory, the chromatic polynomial, introduced by Birkhoff and Lewis \cite{BL} in order to treat the four color theorem,
is a polynomial invariant attached to a graph; its values at $X=k$ gives the number of valid colorings of the graph with $k$ colors,
for any integer $k\geq 1$. Numerous results are known on this object, as for example the alternation of signs of its coefficients, a result due to Rota \cite{Rota},
proved with the help of the Möbius inversion in certain lattices.

Our aim here is to insert chromatic polynomials into the theory of combinatorial Hopf algebras, and to recover new proofs of these classical results.
Our main tools, presented in the first section, will be a Hopf algebra  $(\hgr,m,\Delta)$ and a bialgebra $(\hgr,m,\delta)$, both based on graphs. 
They share the same product, given by disjoint union; the first (cocommutative) coproduct, denoted by $\Delta$, 
is given by partitions of vertices  into two parts;  the second (not cocommutative) one, denoted by $\delta$, is given by a contraction-extraction process. 
For example:
\begin{align*}
\Delta(\grtroisun)&=\grtroisun\otimes 1+1\otimes \grtroisun+3\grdeux\otimes \grun+3\grun \otimes \grdeux,\\
\delta(\grtroisun)&=\grun \otimes \grtroisun+3\grdeux\otimes \grun\grdeux+\grtroisun\otimes \grun\grun\grun,
\end{align*}
or, in a decorated version, where $a,b,c$ are positive integers:
\begin{align*}
\Delta(\grdtroisun{$a$}{$c$}{$b$})&=\grdtroisun{$a$}{$c$}{$b$}\otimes 1+1\otimes \grdtroisun{$a$}{$c$}{$b$}
+\grddeux{$a$}{$b$}\otimes \grdun{$c$}+\grddeux{$a$}{$c$}\otimes \grdun{$b$}+\grddeux{$b$}{$c$}\otimes \grdun{$a$}
+\grdun{$c$} \otimes \grddeux{$a$}{$b$}+\grdun{$b$} \otimes \grddeux{$a$}{$c$}+\grdun{$a$} \otimes \grddeux{$b$}{$c$},\\
\delta(\grdtroisun{$a$}{$c$}{$b$})&=\grdunter{$a+b+c$} \otimes \grdtroisun{$a$}{$c$}{$b$}
+\grddeuxbis{$a+b$}{$c$}\otimes \grdun{$c$}\grddeux{$a$}{$b$}
+\grddeuxbis{$a+c$}{$b$}\otimes \grdun{$b$}\grddeux{$a$}{$c$}
+\grddeuxbis{$b+c$}{$a$}\otimes \grdun{$a$}\grddeux{$b$}{$c$}
+\grdtroisun{$a$}{$c$}{$b$}\otimes\grdun{$a$}\grdun{$b$}\grdun{$c$}.
\end{align*}
We obtain a Hopf algebra $(\hgr,m,\Delta)$, graded by the cardinality of graphs, and connected, 
that is to say its connected component of degree $0$
is reduced to the base field $\Q$: this is what is usually called a \emph{combinatorial Hopf algebra}. On the other side, $(\hgr,m,\delta)$
is a bialgebra, graded by the degree defined by:
\[\deg(G)=\sharp\{\mbox{vertices of $G$}\}-\sharp\{\mbox{connected components of $G$}\}.\]
These two bialgebras are in cointeraction, a notion described in \cite{CEFM,FoissyEhrhart,Manchon,Manchon2}:
 $(\hgr,m,\Delta)$ is a bialgebra-comodule over
$(\hgr,m,\delta)$, see Theorem \ref{theo7}. Another example of interacting bialgebras is the pair $(\Q[X],m,\Delta)$ and $(\Q[X],m,\delta)$,
where $m$ is the usual product of $\Q[X]$ and the two coproducts $\Delta$ and $\delta$ are defined by:
\begin{align*}
\Delta(X)&=X\otimes 1+1\otimes X,&\delta(X)&=X\otimes X.
\end{align*}

This has interesting consequences, proved and used on quasi-posets in \cite{FoissyEhrhart}, listed here in Theorem \ref{theo8}. In particular:
\begin{enumerate}
\item There exists a unique morphism $\phi_1:\hgr\rightarrow\Q[X]$, which is a Hopf algebra morphism
from $(\hgr,m,\Delta)$ to $(\Q[X],m,\Delta)$ and also a bialgebra morphism
from $(\hgr,m,\delta)$ to $(\Q[X],m,\delta)$.
\item We denote by $(\mgr,*)$ the monoid of characters of $(\hgr,m,\delta)$. This monoid acts on the set
$E_{\hgr\longrightarrow\Q[X]}$ of Hopf algebra morphisms from $(\hgr,m,\Delta)$ to $(\Q[X],m,\Delta)$, via the map:
\begin{align*}
\leftarrow&:\left\{\begin{array}{rcl}
E_{\hgr\rightarrow\Q[X]}\times \mgr&\longrightarrow&E_{\hgr\rightarrow\Q[X]}\\
(\phi,\lambda)&\longrightarrow&\phi\leftarrow\lambda=(\phi \otimes \lambda)\circ \delta.
\end{array}\right.
\end{align*}
Moreover, the action of $\mgr$ over $E_{\hgr\rightarrow\Q[X]}$ is free of rank 1, with $\phi_1$ as a generator.
\end{enumerate}
The morphism $\phi_1$ is described in the second section: for any graph $G$, $\phi_1(G)$ is the chromatic polynomial $P_{chr}(G)$ 
(Theorem \ref{theo11}). This characterizes the chromatic polynomial as the unique polynomial invariant on graphs
compatible with both bialgebraic structures. 
To this morphim is attached a character denoted by $\lambda_{chr}$, which allows to reconstruct $P_{chr}$
trough the action of $\mgr$: for any graph $G$,
\[P_{chr}(G)=\sum_{\sim} \lambda_{chr}(G|\sim)X^{cl(\sim)},\]
where the sum is over a family of equivalences $\sim$ on the set of vertices of $G$, $cl(\sim)$ 
is the number of equivalence classes of $\sim$,
and $G|\sim$ is a graph obtained by restricting $G$ to the classes of $\sim$ (Corollary \ref{cor12}).
Moreover, the inverse of $\lambda_{chr}$ for the convolution associated to the coproduct $\delta$ is the character $\lambda_0$
which sends any graph to $1$, whereas the inverse of $\lambda_{chr}$  for the convolution associated to the coproduct $\delta$ is
related to acyclic orientations of graphs and allows to describe the antipode of $(\hgr,m,\Delta)$, see Corollary \ref{corantipode}. 
Therefore, the knowledge of the chromatic character implies the knowledge of the chromatic polynomial;
we give a formula for computing this chromatic character on any graph with the notion (used in Quantum Field Theory) of forests,
through the antipode of a quotient of $(\hgr,m,\delta)$, see Proposition \ref{prop13}.
We give a Hopf-algebraic proof of the classical way to compute the chromatic polynomial  by induction on the number of edges
by an extraction-contraction of an edge in Proposition \ref{prop15}, and deduce a similar way to compute
the chromatic character.
As consequences, we obtain proofs of Rota's result on the sign of the coefficients of a chromatic polynomial (Corollary \ref{cor19})
and of Stanley's interpretation of values at negative integers of a chromatic polynomial in Corollary \ref{cor24}.
The link with Rota's proof is made via the lattice attached to a graph, defined in Proposition \ref{prop16}.\\

We then study morphisms from this double bialgebra of graphs to the function of quasisymmetric functions $\QSym$
\cite{Aguiar2,Gelfand,Hazewinkel,MR3,Stanleyenu}. For this, we need to generalize the construction on graphs to graphs decorated
by elements of an abelian semigroup $(\calD,+)$, obtaining a Hopf algebra $(\hgrd,m,\Delta)$
and a bialgebra $(\hgrd,m,\delta)$ on the same algebra, which we give a graduation with the help of a map
$\wt:(\calD,+)\longrightarrow \N_{>0}$ (Proposition \ref{propgrad}).
Using Aguiar, Bergeron and Sottile's theory of combinatorial Hopf algebras \cite{Aguiar2}, 
we introduce a homogeneous Hopf algebra morphism $F_{chr}^{(\calD)}$ from $\hgrd$ to $\QSym$.
 This morphism $F_{chr}^{(\calD)}$ sends any graph $G$ to its 
chromatic symmetric function, as defined by Stanley \cite{Stanley2}. 
Moreover, $F_{chr}^{(\calD)}$ is the unique potential homogeneous morphism compatible with both bialgebraic structure
and we prove that it indeed satisfies this property if, and only if the map 
$\wt$ giving the graduation is a semigroup morphism (Theorems \ref{theo27} and \ref{theo28}).
Note that this condition excludes the nondecorated case, identified with $\calD=\{*\}$, giving its unique semigroup structure
$*+*=*$ and the graduation defined by $\wt(*)=1$.
As a consequence, we obtain a diagram of Hopf algebra morphisms:
\[\xymatrix{\hgrd\ar^{P_{chr}^{(\calD)}}[r] \ar_{F_{chr}^{(\calD)}}[d]&\Q[X]\\
\QSym \ar[ru]_{H}&}\]
where $H$ is given with the help of Hilbert polynomials (Proposition \ref{prop25}).
The last section deals with a non-commutative version of the chromatic symmetric function: the Hopf algebra of graphs
is replaced by a non-commutative Hopf algebra of indexed graphs, and $\QSym$ 
is replaced by the Hopf algebra of packed words $\WQSym$. For any indexed graph $G$, its non-commutative chromatic 
symmetric function $\bfF_{chr}(G)$ can also be seen as a symmetric formal series
in non-commutative indeterminates  (Theorem \ref{theo35}): we recover in this way Gebhard and Sagan's chromatic symmetric function  
introduced in \cite{Gebhard} and related in \cite{Rosas} to MacMahon symmetric functions.\\
%We also  obtain in this way a non-commutative version $\bfF_0$ of $F_0$ (Proposition \ref{prop36}), 
%also related to $\bfF_{chr}$ by the action of the character $\lambda_{chr}$.\\ 

\textbf{Acknowledgment.} 
The research leading these results was partially supported by the French National Research Agency under the reference
ANR-12-BS01-0017.\\

\textbf{Thanks}. I am grateful to Mercedes Rosas, who pointed the link with Gebhard and Sagan's 
chromatic symmetric function in noncommuting variables, and to Viviane Pons, who noticed an important mistake in the preceding
version of the paper. \\

\begin{notation}
\begin{enumerate}
\item All the vector spaces in the text are taken over $\Q$.
\item We denote by $\N_{>0}=\{1,2,3,\ldots\}$ the set of positive integers.
\item For any integer $n \geq 0$, we denote by $[n]$ the set $\{1,\ldots,n\}$. In particular, $[0]=\emptyset$.
\item The usual product of the polynomial algebra $\Q[X]$ is denoted by $m$. 
This algebra is given two bialgebra structures, defined by:
\begin{align*}
\Delta(X)&=X\otimes 1+1\otimes X,&
\delta(X)&=X\otimes X.
\end{align*}
Identifying $\Q[X,Y]$ and $\Q[X]\otimes \Q[Y]$:
\begin{align*}
&\forall P\in \Q[X],&\Delta(P)(X,Y)&=P(X+Y),&\delta(P)(X,Y)&=P(XY).
\end{align*}
The counit of $\Delta$ is given by:
\begin{align*}
&\forall P\in \Q[X],&\varepsilon(P)&=P(0).
\end{align*}
The counit of $\delta$ is given by:
\begin{align*}
&\forall P\in \Q[X],&\varepsilon'(P)&=P(1).
\end{align*}
Moreover, $(\Q[X],m,\Delta)$ is a Hopf algebra, of antipode $S$ sending any $P(X)\in \Q[X]$ to $P(-X)$. 
\end{enumerate}\end{notation}

\section{Hopf algebraic structures on graphs}

\label{s1}

We refer to \cite{Harary} for classical results and vocabulary on graphs. Recall that a graph is a pair $G=(V(G),E(G))$, 
where $V(G)$ is a finite set, and $E(G)$ is a subset of the set of parts of $V(G)$ of cardinality $2$.
In sections \ref{s1} and \ref{s2},  we  shall work with isoclasses of graphs, which we will simply call graphs.
For any graph $G$, we denote by $|G|$  the cardinality of $V(G)$ 
and by $cc(G)$ the number of its connected components. By convention, the empty graph $1$ is considered as non connected.
The set of graphs is denoted by $\gr$. For example, here are graphs $G$ with $|G|\leq 4$:
\begin{align*}
&1;&&\grun;&&\grdeux,\grun\grun;&&\grtroisun,\grtroisdeux,\grdeux\grun,\grun\grun\grun;&&
\grquatreun,\grquatredeux,\grquatretrois,\grquatrequatre,\grquatrecinq,\grquatresix,
\grtroisun\grun,\grtroisdeux\grun,\grdeux\grdeux,\grdeux\grun,\grun\grun\grun\grun.
\end{align*}
A graph is \emph{totally disconnected} if it has no edge.\\

We denote by $\hgr$ the vector space generated by the set of graphs. The disjoint union of graphs gives it a commutative,
associative product $m$. As an algebra, $\hgr$ is (isomorphic to) the free commutative algebra generated by connected graphs.

\subsection{The first coproduct}

\begin{defi}
Let $G$ be a graph and $I\subseteq V(G)$. The graph $G_{\mid I}$ is defined by:
\begin{itemize}
\item $V(G_{\mid I})=I$.
\item $E(G_{\mid I})=\{ \{x,y\}\in E(G)\mid x,y\in I\}$.
\end{itemize}\end{defi}

We refer to \cite{Abe,Kassel,Sweedler} for classical results and notations on bialgebras and Hopf algebras.
The following Hopf algebra is introduced in \cite{Schmitt}:

\begin{prop} \label{prop2}
We define a coproduct $\Delta$ on $\hgr$ by:
\begin{align*}
&\forall G\in \gr,&\Delta(G)&=\sum_{V(G)=I\sqcup J} G_{\mid I}\otimes G_{\mid J}.
\end{align*}
Then $(\hgr,m,\Delta)$ is a graded, connected, cocommutative Hopf algebra.
Its counit is given by:
\begin{align*}
&\forall G\in \gr,&\varepsilon(G)&=\delta_{G,1}.
\end{align*}\end{prop}

\begin{proof} If $G,H$ are two graphs, then $V(GH)=V(G)\sqcup V(H)$, so:
\begin{align*}
\Delta(GH)&=\sum_{\substack{V(G)=I\sqcup J,\\V(H)=K\sqcup L}} GH{\mid I\sqcup K} \otimes GH_{\mid J\sqcup L}\\
&=\sum_{\substack{V(G)=I\sqcup J,\\V(H)=K\sqcup L}} G_{\mid I} H_{\mid K}\otimes G_{\mid J} H_{\mid L}\\
&=\Delta(G)\Delta(H).
\end{align*}
If $G$ is a graph, and $I\subseteq J\subseteq V(G)$, then $(G_{\mid I})_{\mid J}=G_{\mid J}$. Hence:
\begin{align*}
(\Delta \otimes \id)\circ \Delta(G)&=\sum_{\substack{V(G)=I\sqcup L,\\ I=J\sqcup K}} (G_{\mid I})_{\mid J}\otimes (G_{\mid I})_{\mid K}
\otimes G_{\mid L}\\
&=\sum_{V(G)=J\sqcup K\sqcup L} G_{\mid J}\otimes G_{\mid K}\otimes G_{\mid L}\\
&=\sum_{\substack{V(G)=J\sqcup I,\\ I=K\sqcup L}} G_{\mid J} \otimes (G_{\mid I})_{\mid K}\otimes (G_{\mid I})_{\mid L}\\
&=(\id \otimes \Delta)\circ \Delta(G).
\end{align*}
So $\Delta$ is coassociative. It is obviously cocommutative. \end{proof}

\begin{example}
\begin{align*}
\Delta(\grun)&=\grun \otimes 1+1\otimes \grun,\\
\Delta(\grdeux)&=\grdeux\otimes 1+1\otimes \grdeux+2\grun \otimes \grun,\\
\Delta(\grtroisun)&=\grtroisun\otimes 1+1\otimes \grtroisun+3\grdeux\otimes \grun+3\grun \otimes \grdeux,\\
\Delta(\grtroisdeux)&=\grtroisdeux\otimes 1+1\otimes \grtroisdeux+2\grdeux \otimes \grun+\grun\grun\otimes \grun+
2\grun \otimes \grdeux+\grun\otimes \grun\grun.
\end{align*}\end{example}

\subsection{The second coproduct}

\begin{notation}
Let $V$ be a finite set $\sim$ be an equivalence on $V$. 
\begin{itemize}
\item We denote by $\pi_\sim:V\longrightarrow V/\sim$ the canonical surjection.
\item We denote by $cl(\sim)$ the cardinality of $V/\sim$.
\end{itemize}\end{notation}

\begin{defi}
Let $G$ a graph, and $\sim$ be an equivalence relation on $V(G)$. 
\begin{enumerate}
\item (Contraction). The graph $V(G)/\sim$ is defined by:
\begin{align*}
V(G/\sim)&=V(G)/\sim,\\
E(G/\sim)&=\{\{\pi_\sim(x),\pi_\sim(y)\}\mid \{x,y\}\in E(G),\: \pi_\sim(x)\neq \pi_\sim(y)\}.
\end{align*}
\item (Extraction). The graph $V(G)|\sim$ is defined by:
\begin{align*}
V(G|\sim)&=V(G),\\
E(G|\sim)&=\{\{x,y\}\in E(G) \mid  x\sim y\}.
\end{align*}
\item We shall write $\sim \triangleleft G$ if, for any $c\in V(G)/\sim$, $G_{\mid c}$ is connected.
\end{enumerate}\end{defi}

Roughly speaking, $G/\sim$ is obtained by contracting each equivalence class of $\sim$ to a single vertex, and by deleting
the loops and multiple edges created in the process; $G\mid \sim$ is obtained by deleting the edges which extremities are not equivalent,
so is the product of the restrictions of $G$ to the equivalence classes of $\sim$. \\

We now define a coproduct on $\hgr$. This coproduct, which can also be found in \cite{Schmitt}, can also be deduced from a 
general operadic construction \cite{vdLaan}, see also \cite{Aguiar}.
A similar construction is defined on various families of oriented graphs in \cite{Manchon}. 

\begin{prop} \label{prop4}
We define a coproduct $\delta$ on $\hgr$ by:
\begin{align*}
&\forall G\in \gr,& \delta(G)&=\sum_{\sim\triangleleft G} (G/\sim) \otimes (G|\sim).
\end{align*}
Then $(\hgr,m,\delta)$ is a bialgebra.
Its counit is given by:
\begin{align*}
&\forall G\in \gr,&\varepsilon'(G)&=\begin{cases}
1\mbox{ if $G$ is totally disconnected},\\
0\mbox{ otherwise}.
\end{cases}
\end{align*} 
It is graded, putting:
\begin{align*}
&\forall G\in \gr,& \deg(G)&=|G|-cc(G).
\end{align*}
In particular, a basis of its homogeneous component of degree $0$ is given by totally disconnected graphs, including $1$. \end{prop}

\begin{proof}
Let $G$, $H$ be graphs and $\sim$ be an equivalence on $V(GH)=V(G)\sqcup V(H)$. We put $\sim'=\sim_{\mid V(G)}$
and $\sim''_{\mid V(H)}$. The connected components of $GH$ are the ones of $G$ and $H$, so
 $\sim\triangleleft GH$ if, and only if, the two following conditions are satisfied:
\begin{itemize}
\item $\sim'\triangleleft G$ and $\sim''\triangleleft H$.
\item If $x\sim y$, then $(x,y)\in V(G)^2\sqcup V(H)^2$ .
\end{itemize}
Note that the second point implies that $\sim$ is entirely determined by $\sim'$ and $\sim''$.
Moreover, if this holds, $(GH)/\sim=(G/\sim')(H/\sim'')$ and $(GH)|\sim=(G|\sim')(H|\sim'')$, so:
\begin{align*}
\delta(GH)&=\sum_{\sim'\triangleleft G,\:\sim''\triangleleft H}(G/\sim')(H/\sim'')\otimes (G|\sim')(H|\sim'')
=\delta(G)\delta(H).
\end{align*}

Let $G$ be a graph. If $\sim\triangleleft G$, the connected components of $G/\sim$ are the image by the canonical surjection
of the connected components of $G$; the connected components of $G|\sim$ are the equivalence classes of $\sim$.
If $\sim$ and $\sim'$ are two equivalences on $G$, we shall denote $\sim' \leq \sim$ if for all $x,y\in V(G)$,
$x\sim' y$ implies $x\sim y$. Then:
\begin{align*}
(\delta \otimes \id)\circ \delta(G)&=\sum_{\sim\triangleleft G,\sim'\triangleleft G/\sim} (G/\sim)/\sim'\otimes (G/\sim)|\sim'\otimes G|\sim\\
&=\sum_{\substack{\sim,\sim'\triangleleft G,\\ \sim'\leq \sim}} (G/\sim)/\sim'\otimes (G/\sim)|\sim' \otimes G|\sim\\
&=\sum_{\substack{\sim,\sim'\triangleleft G,\\ \sim'\leq \sim}} (G/\sim')\otimes (G|\sim')/\sim\otimes (G|\sim')|\sim\\
&=\sum_{\sim\triangleleft G,\sim' \triangleleft G|\sim} (G/\sim')\otimes (G|\sim')/\sim\otimes (G|\sim')|\sim\\
&=(\id \otimes \delta)\circ \delta(G).
\end{align*}
So $\delta$ is coassociative.\\

We define two special equivalence relations $\sim_0$ and $\sim_1$ on $G$: for all $x,y\in V(G)$,
\begin{itemize}
\item $x\sim_0 y$ if, and only if, $x=y$.
\item $x\sim_1 y$ if, and only if, $x$ and $y$ are in the same connected component of $G$.
\end{itemize}
Note that $\sim_0$, $\sim_1 \triangleleft G$. Moreover, if $\sim \triangleleft G$, $G/\sim$ is not totally disconnected, except if $\sim=\sim_1$;
$G|\sim$ is not totally disconnected, except if $\sim=\sim_0$. Hence:
\begin{itemize}
\item If $G$ is totally disconnected, then $\delta(G)=G\otimes G$.
\item Otherwise, putting $n=|G|$ and $k=cc(G)$:
\[\delta(G)=\grun^k \otimes G+G\otimes \grun^n+\ker(\varepsilon')\otimes \ker(\varepsilon').\]
\end{itemize}
So $\varepsilon'$ is indeed the counit of $\delta$. \\

Let $G$ be a graph, with $n$ vertices and $k$ connected components (so of degree $n-k$). 
Let $\sim\triangleleft G$.
Then:
\begin{enumerate}
\item $G/\sim$ has cardinality $cl(\sim)$ and $k$ connected components, so is of degree $cl(\sim)-k$.
\item $G|\sim$ has cardinality $n$ and $cl(\sim)$ connected components, so is of degree $n-cl(\sim)$.
\end{enumerate}
Hence, $\deg(G/\sim)+\deg(G|\sim)=cl(\sim)-k+n-cl(\sim)=n-k=\deg(G)$: $\delta$ is homogeneous. \end{proof}

\begin{example}
\begin{align*}
\delta(\grun)&=\grun\otimes \grun,&
\delta(\grdeux)&=\grun \otimes \grdeux+\grdeux\otimes \grun\grun,\\
\delta(\grtroisun)&=\grun \otimes \grtroisun+3\grdeux\otimes \grun\grdeux+\grtroisun\otimes \grun\grun\grun,&
\delta(\grtroisdeux)&=\grun \otimes \grtroisdeux+2\grdeux\otimes \grun\grdeux+\grtroisdeux\otimes \grun\grun\grun.
\end{align*}\end{example}

\begin{remark}
 Let $G\in\gr$. The following conditions are equivalent:
\begin{itemize}
\item $\varepsilon'(G)=1$.
\item $\varepsilon'(G)\neq 0$.
\item $\deg(G)=0$.
\item $G$ is totally disconnected.
\end{itemize}\end{remark}

\subsection{Antipode for the second coproduct}

$(\hgr,m,\delta)$ is not a Hopf algebra: the group-like element $\grun$ has no inverse. However, the graduation of $(\hgr,m,\delta)$ induced
a graduation of $\hgr'=(\hgr,m,\delta)/\langle \grun-1\rangle$, which becomes a graded, connected bialgebra, hence a Hopf algebra;
we denote its antipode by $S'$. Note that, as a commutative algebra, $\hgr'$ is freely generated by connected graphs different from $\grun$. \\

The notations and ideas of the following definition and theorem come from Quantum Field Theory, where they are applied
to Renormalization with the help of Hopf algebras of Feynman graphs; see for example \cite{CK1,CK2} for an introduction.

\begin{defi} 
Let $G$ be a connected graph, $G\neq \grun$. 
\begin{enumerate}
\item A \emph{forest} of $G$ is a set $\calf$ of subsets of $V(G)$, such that:
\begin{enumerate}
\item $V(G)\in \calf$.
\item If $I,J\in \calf$, then $I\subseteq J$, or $J\subseteq I$, or $I\cap J=\emptyset$.
\item For all $I\in \calf$, $G_{\mid I}$ is connected and not reduced to the graph $\grun$. 
\end{enumerate}
The set of forests of $G$ is denoted by $\F(G)$.
\item Let $\calf\in \F(G)$; it is partially ordered by the inclusion. 
For any $I\in \F(G)$, the relation $\sim_I$ is the equivalence on $I$ which classes are the maximal elements (for the inclusion)
of $\{J\in \calf\mid J\subsetneq I\}$ (if this is non-empty), and singletons. We put:
\[G_\calf=\prod_{I\in \calf} (G_{\mid I})/\sim_I.\]
\end{enumerate}\end{defi}

\begin{example}
The graph $\grdeux$ has only one forest, $\calf=\{\grdeux\}$;  $\grdeux_\calf=\grdeux$.
The graph $\grtroisun$ has four forests:
\begin{itemize}
\item $\calf=\{\grtroisun\}$; in this case, $\grtroisun_\calf=\grtroisun$.
\item Three forests $\calf=\{\grtroisun,\grdeux\}$; for each of them, $\grtroisun_\calf=\grdeux\grdeux$.
\end{itemize}\end{example}

\begin{theo}
For any connected graph $G$, $G\neq \grun$, in $\hgr'$:
\[S'(G)=\sum_{\calf \in \F(G)} (-1)^{\sharp\calf} G_\calf.\]
\end{theo}

\begin{proof} By induction on the number $n$ of vertices of $G$. If $n=2$, then $G=\grdeux$.
As $\delta'(\grdeux)=\grdeux\otimes 1+1\otimes \grdeux$, $S'(\grdeux)=-\grdeux=-\grdeux_\calf$, where $\calf=\{\grdeux\}$ 
is the unique forest of $\grdeux$. Let us assume the result at all ranks $<n$. Then:
\begin{align*}
S'(G)&=-G-\sum_{\sim \triangleleft G,\: \sim\neq \sim_1} (G/\sim) S'(G|\sim)\\
&=-G-\sum_{\substack{\sim \triangleleft G,\: \sim\neq \sim_1\\ G/\sim=\{I_1,\ldots,I_k\}}} \sum_{\calf_i \in \F(G_{\mid I_i})}
(-1)^{\sharp\calf_1+\ldots+\sharp\calf_k}(G/\sim) (G_{\mid I_1})_{\calf_1}\ldots (G_{\mid I_k})_{\calf_k}.
\end{align*}
Note that each forest of $G$ different from $\{G\}$ consists of $\{G\}$ with the union of of forests
$\calf_1,\ldots,\calf_k$ on nonintersecting, connected subsets $I_1,\ldots,I_k$ of $V(G)$. Therefore:
\begin{align*}
S'(G)&=-G-\sum_{\calf\in \F(G),\:\calf\neq \{G\}} (-1)^{\sharp\calf-1}G_\calf
=\sum_{\calf \in \F(G)} (-1)^{\sharp\calf} G_\calf. \qedhere
\end{align*}
%For the third equality, $\calf=\{G\}\sqcup \calf_1\sqcup\ldots\sqcup \calf_k$. 
\end{proof}

\begin{example}
In $\hgr'$:
\begin{align*}
S'(\grdeux)&=-\grdeux,&
S'(\grtroisun)&=-\grtroisun+3\grdeux\grdeux,&
S'(\grtroisdeux)&=-\grtroisdeux+2\grdeux\grdeux.
\end{align*}
\end{example}

\subsection{Cointeraction}

\begin{theo} \label{theo7}
With the coaction $\delta$, $(\hgr,m,\Delta)$  and $(\hgr,m,\delta)$ are in cointeraction,
that is to say that $(\hgr,m,\Delta)$ is a $(\hgr,m,\delta)$-comodule bialgebra, or a Hopf algebra in the category of $(\hgr,m,\delta)$-comodules.
In other words:
\begin{itemize}
\item $\delta(1)=1\otimes 1$.
\item $m_{1,3,24}\circ (\delta\otimes \delta)\circ \Delta=(\Delta\otimes \id)\circ \delta$, with:
\[m_{1,3,24}:\left\{\begin{array}{rcl}
\hgr\otimes \hgr\otimes \hgr\otimes \hgr&\longrightarrow&\hgr\otimes \hgr\otimes \hgr\\
a_1\otimes b_1 \otimes a_2 \otimes b_2&\longrightarrow&a_1 \otimes a_2 \otimes b_1 b_2.
\end{array}\right.\]
\item For all $a,b\in \hgr$, $\delta(ab)=\delta(a)\delta(b)$.
\item For all $a\in \hgr$, $(\varepsilon\otimes \id)\circ \delta(a)=\varepsilon(a)1$.
\end{itemize}\end{theo}

\begin{proof} The first and third points are already proved, and the fourth one is immediate for any $a\in \gr$. Let us prove the second point.
For any graph $G\in \gr$:
\begin{align*}
(\Delta \otimes \id)\circ \delta(G)&=\sum_{\sim\triangleleft G,\:V(G)/\sim=I\sqcup J} (G/\sim)_{\mid I} \otimes (G/\sim)_{\mid J} \otimes G|\sim\\
&=\sum_{\substack{V(G)=I'\sqcup J',\\\sim'\triangleleft G_{\mid I}, \:\sim''\triangleleft G_{\mid J}}}
(G_{\mid I'})/\sim'\otimes (G_{\mid J'})/\sim''\otimes (G_{\mid I'})|\sim' (G_{\mid J'})|\sim''\\
&=m_{1,3,24}\circ (\delta\otimes \delta)\circ \Delta(G).
\end{align*}
For the second equality, $I'=\pi_\sim^{-1}(I)$, $I''=\pi_\sim^{-1}(J)$, $\sim'=\sim_{\mid I'}$ and $\sim''=\sim_{\mid J'}$.  \end{proof}

\subsection{Decorated versions}

We fix a nonempty set $\calD$. 

\begin{defi}
A $\calD$-decorated graph is a pair $(G,d_G)$, where $G$ is a graph and $d_G:V(G)\longrightarrow\calD$ is a map.
We denote by $\gr(\calD)$ the set of isoclasses of $\calD$-decorated graphs, 
and by $\hgrd$ the vector space generated by $\gr(\calD)$. 
\end{defi}

\begin{example}
For any $k\in \N$, let us denote by $\gr_k(\calD)$ the set of $\calD$-decorated graphs with $k$ vertices.
Then:
\begin{align*}
\gr_1(\{a,b,c\})&=\{\grdun{$a$},\grdun{$b$},\grdun{$c$}\},\\
\gr_2(\{a,b,c\})&=\{\grdun{$a$}\grdun{$a$ },\grdun{$a$}\grdun{$b$ },\grdun{$a$}\grdun{$c$ },
\grdun{$b$}\grdun{$b$ },\grdun{$b$}\grdun{$c$ },\grdun{$c$}\grdun{$c$ },
\grddeux{$a$}{$a$ },\grddeux{$a$}{$b$ },\grddeux{$a$}{$c$ },\grddeux{$b$}{$b$ },\grddeux{$b$}{$c$ },\grddeux{$c$}{$c$ }
\},\\
\gr_3(\{a,b,c\})&=\left\{\begin{array}{c}
\grdun{$a$}\grdun{$a$}\grdun{$a$},
\grdun{$a$}\grdun{$a$}\grdun{$b$},\grdun{$a$}\grdun{$a$}\grdun{$c$},
\grdun{$a$}\grdun{$b$}\grdun{$b$},\grdun{$a$}\grdun{$b$}\grdun{$c$},\grdun{$a$}\grdun{$c$}\grdun{$c$},
\grdun{$b$}\grdun{$b$}\grdun{$b$},\grdun{$b$}\grdun{$b$}\grdun{$c$},
\grdun{$b$}\grdun{$c$}\grdun{$c$},\grdun{$c$}\grdun{$c$}\grdun{$c$},\\
\grdun{$a$}\grddeux{$a$ }{$a$},\grdun{$a$}\grddeux{$a$ }{$b$},\grdun{$a$}\grddeux{$a$ }{$c$},
\grdun{$a$}\grddeux{$b$ }{$b$},\grdun{$a$}\grddeux{$b$ }{$c$},\grdun{$a$}\grddeux{$c$ }{$c$},
\grdun{$b$}\grddeux{$a$ }{$a$},\grdun{$b$}\grddeux{$a$ }{$b$},\grdun{$b$}\grddeux{$a$ }{$c$},
\grdun{$b$}\grddeux{$b$ }{$b$},\grdun{$b$}\grddeux{$b$ }{$c$},\grdun{$b$}\grddeux{$c$ }{$c$},\\
\grdun{$c$}\grddeux{$a$ }{$a$},\grdun{$c$}\grddeux{$a$ }{$b$},\grdun{$c$}\grddeux{$a$ }{$c$},
\grdun{$c$}\grddeux{$b$ }{$b$},\grdun{$c$}\grddeux{$b$ }{$c$},\grdun{$c$}\grddeux{$c$ }{$c$},
\grdtroisdeux{$a$}{$a$}{$a$},\grdtroisdeux{$a$}{$b$}{$a$},\grdtroisdeux{$a$}{$c$}{$a$},
\grdtroisdeux{$a$}{$b$}{$b$},\grdtroisdeux{$a$}{$c$}{$b$},\grdtroisdeux{$a$}{$c$}{$c$},\\
\grdtroisdeux{$b$}{$a$}{$a$},\grdtroisdeux{$b$}{$b$}{$a$},\grdtroisdeux{$b$}{$c$}{$a$},
\grdtroisdeux{$b$}{$b$}{$b$},\grdtroisdeux{$b$}{$c$}{$b$},\grdtroisdeux{$b$}{$c$}{$c$},
\grdtroisdeux{$c$}{$a$}{$a$},\grdtroisdeux{$c$}{$b$}{$a$},\grdtroisdeux{$c$}{$c$}{$a$},
\grdtroisdeux{$c$}{$b$}{$b$},\grdtroisdeux{$c$}{$c$}{$b$},\grdtroisdeux{$c$}{$c$}{$c$},\\
\grdtroisun{$a$}{$a$}{$a$},\grdtroisun{$a$}{$b$}{$a$},\grdtroisun{$a$}{$c$}{$a$},
\grdtroisun{$a$}{$b$}{$b$},\grdtroisun{$a$}{$c$}{$b$},\grdtroisun{$a$}{$c$}{$c$},
\grdtroisun{$b$}{$b$}{$b$},\grdtroisun{$b$}{$c$}{$b$},\grdtroisun{$b$}{$c$}{$c$},
\grdtroisun{$c$}{$c$}{$c$}
\end{array}\right\}.
\end{align*}
\end{example}

If $G$ and $H$ are two $\calD$-decorated graphs, their disjoint union is naturally also a $\calD$-decorated graph:
hence, the disjoint union makes $\hgrd$ an associative, commutative algebra, which unit is the empty graph $1$.
Moreover, if $G$ is a $\calD$-decorated graph and $I\subset V(G)$, then $G_{\mid I}$ is also a $\calD$-decorated graph,
with $d_{G_{\mid I}}=(d_G)_{\mid I}$. Then $\hgrd$ is a Hopf algebra, with the coproduct defined by:
\begin{align*}
&\forall G\in \gr(\calD),&\Delta(G)&=\sum_{V(G)=I\sqcup J} G_{\mid I}\otimes G_{\mid J}.
\end{align*}

\begin{example} If $a,b,c\in \calD$:
\begin{align*}
\Delta(\grdun{$a$})&=\grdun{$a$} \otimes 1+1\otimes \grdun{$a$},\\
\Delta(\grddeux{$a$}{$b$})&=\grddeux{$a$}{$b$}\otimes 1+1\otimes \grddeux{$a$}{$b$}+\grdun{$a$} \otimes \grdun{$b$}
+\grdun{$b$} \otimes \grdun{$a$},\\
\Delta(\grdtroisun{$a$}{$c$}{$b$})&=\grdtroisun{$a$}{$c$}{$b$}\otimes 1+1\otimes \grdtroisun{$a$}{$c$}{$b$}
+\grddeux{$a$}{$b$}\otimes \grdun{$c$}+\grddeux{$a$}{$c$}\otimes \grdun{$b$}+\grddeux{$b$}{$c$}\otimes \grdun{$a$}
+\grdun{$c$} \otimes \grddeux{$a$}{$b$}+\grdun{$b$} \otimes \grddeux{$a$}{$c$}+\grdun{$a$} \otimes \grddeux{$b$}{$c$},\\
\Delta(\grdtroisdeux{$a$}{$c$}{$b$})&=\grdtroisdeux{$a$}{$c$}{$b$}\otimes 1+1\otimes\grdtroisdeux{$a$}{$c$}{$b$}
+\grddeux{$a$}{$b$}\otimes \grdun{$c$}+\grddeux{$a$}{$c$}\otimes \grdun{$b$}
+\grdun{$b$}\grdun{$c$}\otimes \grdun{$a$}+\grdun{$b$}\otimes \grddeux{$a$}{$c$}
+\grdun{$c$}\otimes \grddeux{$a$}{$b$}+\grdun{$a$}\otimes \grdun{$b$}\grdun{$c$}.
\end{align*}\end{example}

In order to define the second coproduct, we need more structure on $\calD$: let us assume that $(\calD,+)$
is an abelian semigroup (that is to say, $+$ is a commutative, associative binary operation on $\calD$). 
If $G$ is a $\calD$-decorated graph and $\sim$ is an equivalence on $V(G)$. As $V(G\mid \sim)=V(G)$,
$G\mid \sim$ is a $\calD$-decorated graph, with $d_{G\mid \sim}=d_G$. W define $d_{G/\sim}$ by:
\begin{align*}
&\forall c\in V(G/\sim)=V(G)/\sim,&d_{G/\sim}(c)&=\sum_{x\in c} d_G(x).
\end{align*}
As $(\calD,+)$ is an abelian semigroup, this is well-defined, and in this way $G/\sim$ becomes a $\calD$-decorated graph.
The proof of Proposition \ref{prop4} can be extended to the $\calD$-decorated case;
with the notations of the proof of this theorem, if $\sim$, $\sim'\triangleleft G$ and $\sim'\leq \sim$, then, as decorated graphs:
\begin{align*}
(G/\sim)/\sim'&=(G/\sim'),&
(G/\sim)\mid \sim'&=(G|\sim')/\sim,&
G\mid \sim&=(G\mid \sim')\mid \sim.
\end{align*}
Hence, $\hgrd$ is a bialgebra, with the coproduct defined by:
\begin{align*}
&\forall G\in \gr(\calD),& \delta(G)&=\sum_{\sim\triangleleft G} (G/\sim) \otimes (G|\sim).
\end{align*}

\begin{example} If $a,b,c\in \calD$:
\begin{align*}
\delta(\grdun{$a$})&=\grdun{$a$}\otimes \grdun{$a$},\\
\delta(\grddeux{$a$}{$b$})&=\grdunbis{$a+b$}\otimes \grddeux{$a$}{$b$}
+\grddeux{$a$}{$b$}\otimes \grdun{$a$}\grdun{$b$},\\
\delta(\grdtroisun{$a$}{$c$}{$b$})&=\grdunter{$a+b+c$} \otimes \grdtroisun{$a$}{$c$}{$b$}
+\grddeuxbis{$a+b$}{$c$}\otimes \grdun{$c$}\grddeux{$a$}{$b$}
+\grddeuxbis{$a+c$}{$b$}\otimes \grdun{$b$}\grddeux{$a$}{$c$}
+\grddeuxbis{$b+c$}{$a$}\otimes \grdun{$a$}\grddeux{$b$}{$c$}
+\grdtroisun{$a$}{$c$}{$b$}\otimes\grdun{$a$}\grdun{$b$}\grdun{$c$},\\
\delta(\grdtroisdeux{$a$}{$c$}{$b$})&=\grdunter{$a+b+c$}\otimes \grdtroisdeux{$a$}{$c$}{$b$}
+\grddeuxbis{$a+c$}{$b$}\otimes \grdun{$b$}\grddeux{$a$}{$c$}
+\grddeuxbis{$a+b$}{$c$}\otimes \grdun{$c$}\grddeux{$a$}{$b$}
+\grdtroisdeux{$a$}{$c$}{$b$}\otimes \grdun{$a$}\grdun{$b$}\grdun{$c$}.
\end{align*}\end{example}

\begin{theo} \label{theo7bis}
With the coaction $\delta$, $(\hgrd,m,\Delta)$ and $(\hgrd,m,\delta)$ are in cointeraction.
Moreover, let us consider the forgetful map:
\[\calF^{(\calD)}:\left\{\begin{array}{rcl}
\hgrd&\longrightarrow&\hgr\\
(G,d_G)\in \gr(\calD)&\longrightarrow&G\in \gr,
\end{array}\right.\]
then $\calF^{(\calD)}$ is a surjective Hopf algebra morphism from $(\hgrd,m,\Delta)$ to $(\hgr,m,\Delta)$ 
and a bialgebra morphism from  compatible $(\hgrd,m,\delta)$ to $(\hgr,m,\delta)$.
\end{theo}

\begin{remark}
When $\calD$ is a singleton, $\calF^{(\calD)}$ is an isomorphism.
Through this isomorphism, we identify $\hgr$ with $\hgrd$, when $\calD=\{*\}$ is a singleton,
given its unique semigroup structure $*+*=*$. 
\end{remark}

Let us now give $\hgrd$ a graduation. A graded set is a pair $(\calD,\wt)$, where $\wt:\calD\longrightarrow \N_{>0}$ is a map.
Given such a map, we put, for any $\calD$-decorated graph $G$:
\[\wt(G)=\sum_{x\in V(G)} \wt(x).\]
For any $n \geqslant 0$, let $(\hgrd)_n$ be the subspace of $\hgrd$ generated by the $\calD$-decorated
graphs $G$ with $\wt(G)=n$. Then:

\begin{prop}\label{propgrad}
If $(\calD,\wt)$ is a graded set, the map $\wt$ induces a connected graduation of the Hopf algebra $(\hgrd,m,\Delta)$.
\end{prop}

\begin{remark}
The nondecorated case $\hgr$ is obtained with the weight defined by $\wt(*)=1$. 
\end{remark}

\section{Chromatic polynomials}

\label{s2}

In all this section, we fix an abelian semigroup $(\calD,+)$ and work in $\hgrd$. This situation includes the nondecorated case,
when $\calD=\{*\}$.  

\subsection{Consequence of the cointeraction}

We can apply the results of \cite{FoissyEhrhart}:

\begin{theo}\label{theo8}
We denote by $\mgrd$ the monoid of characters of $\hgrd$. In the nondecorated case, we shall simply write $\mgr$. 
\begin{enumerate}
\item Let $\lambda \in \mgrd$. It is an invertible element if, and only if, for any $d\in \calD$, $\lambda(\grdun{$d$})\neq 0$.
\item Let $B$ be a Hopf algebra, and $E_{\hgrd\rightarrow B}$ 
be the set of Hopf algebra morphisms from $(\hgrd,m,\Delta)$ to $B$. 
Then $\mgrd$ acts on $E_{\hgrd\rightarrow B}$ by:
\[\leftarrow:\left\{\begin{array}{rcl}
E_{\hgrd\rightarrow B}\times \mgrd&\longrightarrow&E_{\hgrd\rightarrow B}\\
(\phi,\lambda)&\longrightarrow&\phi\leftarrow \lambda=(\phi \otimes \lambda)\circ \delta.
\end{array}\right.\]
\item Let $\lambda \in \mgrd$. There exists a unique element $\phi\in E_{\hgrd\rightarrow\Q[X]}$ such that:
\begin{align*}
&\forall x\in \hgrd,&\phi(x)(1)&=\lambda(x).
\end{align*}
\item There exists a unique morphism $\phi_1^{(\calD)}:\hgrd\longrightarrow \Q[X]$, such that:
\begin{itemize}
\item $\phi_1^{(\calD)}$ is a Hopf algebra morphism from $(\hgrd,m,\Delta)$ to $(\Q[X],m,\Delta)$.
\item $\phi_1^{(\calD)}$ is a bialgebra morphism from $(\hgrd,m,\delta)$ to $(\Q[X],m,\delta)$.
\end{itemize}
This morphism is the unique element $\phi\in E_{\hgrd\rightarrow\Q[X]}$ such that:
\begin{align*}
&\forall x\in \hgrd,&\phi(x)(1)&=\varepsilon'(x).
\end{align*}
In the nondecorated case, we shall simply write $\phi_1$. 
\item The following map is a bijection:
\[\left\{\begin{array}{rcl}
\mgrd&\longrightarrow&E_{\hgrd\rightarrow \Q[X]}\\
\lambda&\longrightarrow&\phi_1^{(\calD)}\leftarrow \lambda.
\end{array}\right.\]
\end{enumerate}\end{theo}

We shall determine $\phi_1^{(\calD)}$ in the next section.

\subsection{A first morphism}

\begin{prop} \label{prop9}
We define $\phi_0^{(\calD)}:\hgrd\longrightarrow \Q[X]$ by:
\begin{align*}
&\forall G\in \gr(\calD),&\phi_0^{(\calD)}(G)&=X^{|V(G)|}.
\end{align*}
Then $\phi_0^{(\calD)}$ is a Hopf algebra morphism from $(\hgrd,m,\Delta)$ to $(\Q[X],m,\Delta)$.
In the nondecorated case, we shall simply write $\phi_0$. 
\end{prop}

\begin{proof} This map is obviously an algebra morphism. 
For any graph $G$, of degree $n$:
\begin{align*}
(\phi_0^{(\calD)}\otimes \phi_0^{(\calD)})\circ \Delta(G)
&=\sum_{V(G)=I\sqcup J} X^{|I|}\otimes X^{|J|}=\sum_{i=0}^n \binom{n}{i} X^i \otimes X^{n-i}
=\Delta(X^n)=\Delta\circ \phi_0^{(\calD)}(G).
\end{align*}
So $\phi_0^{(\calD)}$ is a Hopf algebra morphism. \end{proof} 

\begin{remark}
This morphism $\phi_0^{(\calD)}$ is not compatible with $\delta$. For example, in the nondecorated case:
\begin{align*}
\delta \circ \phi_0(\grdeux)&=\delta(X)^2\\
&=X^2\otimes X^2,\\
(\phi_0\otimes \phi_0)\circ \delta(\grdeux)&=(\phi_0\otimes \phi_0)(\grdeux \otimes \grun \grun+\grun \otimes \grdeux)\\
&=X^2\otimes X^2+X\otimes X^2.
\end{align*}
\end{remark}

\subsection{Determination of $\phi_1^{(\calD)}$}

Let us recall the definition of the chromatic polynomial, due to Birkhoff and Lewis \cite{BL}:

\begin{defi} 
Let $G$ be a graph and $X$ a set. 
\begin{enumerate}
\item A $X$-\emph{coloring} of $G$ is a map $f:V(G)\longrightarrow X$.
\item A $\N$-coloring of $G$ is \emph{packed} if $f(V(G))=[k]$, with $k\geq 0$. The set of packed colorings of $G$ is denoted by $\PC(G)$.
\item A \emph{valid} $X$-coloring of $G$ by $X$ is a $X$-coloring $f$ such that if $\{i,j\} \in E(G)$, then $f(i)\neq f(j)$. 
The set of valid $X$-colorings of $G$ is denoted by $\VC(G,X)$; the set of packed valid colorings of $G$ is denoted by $\PVC(G)$.
\item An \emph{independent subset} of $G$ is a subset $I$ of $V(G)$ such that $G_{\mid I}$ is totally disconnected. 
We denote by $\IP(G)$ the set of partitions $\{I_1,\ldots,I_k\}$ of $V(G)$ such that for all $p\in [k]$, $I_p$ is an independent subset of $G$.
\item For any $k\geq 1$, the number of valid $[k]$-colorings of $G$ is denoted by $P_{chr}(G)(k)$. 
This defines a unique polynomial $P_{chr}(G) \in \Q[X]$, called the \emph{chromatic polynomial} of $G$.
\end{enumerate} \end{defi}

Note that if $f$ is a $X$-coloring of a graph $G$, it is valid if, and only if, the partition of $V(G)$ $\{f^{-1}(x)\mid x\in f(V(G))\}$ 
belongs to $\IP(G)$.

\begin{theo} \label{theo11}
\begin{enumerate}
\item The morphism $P_{chr}:\hgr\longrightarrow\Q[X]$ is the morphism $\phi_1$ of Theorem \ref{theo8}.
\item The unique morphism $\phi_1^{(\calD)}$ of Theorem \ref{theo8} is $P_{chr}^{(\calD)}=P_{chr}\circ \calF^{(\calD)}$. 
\end{enumerate}\end{theo}

\begin{proof} 1. It is immediate that, for any graphs $G$ and $H$, $P_{chr}(GH)(k)=P_{chr}(G)(k)P_{chr}(H)(k)$ for any $k$,
so $P_{chr}(GH)=P_{chr}(G)P_{chr}(H)$: $P_{chr}$ is an algebra morphism. Let $G$ be a graph, and $k,l\geq 1$.
We consider the two sets:
\begin{align*}
C&=\VC(G,[k+l]),\\
D&=\{(I,c',c'')\mid I\subseteq V(G), \: c'\in \VC(G_{\mid I},[k]),\: c''\in \VC(G_{\mid V(G)\setminus I},[l])\}.
\end{align*}
We define a map $\theta:C\longrightarrow D$ by $\theta(c)=(I,c',c'')$, with:
\begin{itemize}
\item $I=\{x\in V(G)\mid c(x)\in [k]\}$.
\item For all $x\in I$, $c'(x)=c(x)$.
\item For all $x\notin I$, $c''(x)=c(x)-k$.
\end{itemize}
We define a map $\theta':D\longrightarrow C$ by $\theta(I,c',c'')=c$, with:
\begin{itemize}
\item For all $x\in I$, $c(x)=c'(x)$.
\item For all $x\notin I$, $c(x)=c''(x)+k$.
\end{itemize}
Both $\theta$ and $\theta'$ are well-defined; moreover, $\theta \circ \theta'=\id_D$ and $\theta'\circ \theta=\id_C$,
so $\theta$ is a bijection. Via the identification of $\Q[X]\otimes \Q[X]$ and $\Q[X,Y]$:
\begin{align*}
\Delta \circ P_{chr}(G)(k,l)&=P_{chr}(G)(k+l)\\
&=\sharp C\\
&=\sharp D\\
&=\sum_{I\subseteq V(G)}P_{chr}(G_{\mid I})(k) P_{chr}(G_{\mid V(G)\setminus I})(l)\\
&=(P_{chr}\otimes P_{chr})\left(\sum_{V(G)=I\sqcup J} G_{\mid I} \otimes G_{\mid J}\right)(k,l)\\
&=(P_{chr}\otimes P_{chr})\circ \Delta(G)(k,l).
\end{align*}
As this is true for all $k,l\geq 1$, $\Delta \circ P_{chr}(G)=(P_{chr}\otimes P_{chr})\circ \Delta(G)$.
Moreover:
\[\varepsilon(G)=\varepsilon \circ P_{chr}(G)=P_{chr}(G)(0)=\begin{cases}
1\mbox{ if $G$ is empty},\\
0\mbox{ otherwise}.
\end{cases}\]
So $P_{chr}\in E_{\hgr\rightarrow \Q[X]}$. For any graph $G\in \gr$:
\begin{align*}
P_{chr}(G)(1)&=\begin{cases}
1\mbox{ if $G$ is totally disconnected},\\
0\mbox{ otherwise};
\end{cases}\\
&=\varepsilon'(G).
\end{align*}
So $\phi_1=P_{chr}$. \\

2. By composition, $P_{chr}\circ \calF^{(\calD)}$ satisfies the two required conditions. \end{proof}

\subsection{The chromatic character}

\begin{cor} \label{cor12}
For any connected graph $G\in \gr$, we put:
\[\lambda_{chr}(G)=\frac{dP_{chr}(G)}{dX}(0).\]
We extend $\lambda$ as an element of $\mgr$: for any graph $G$, if $G_1,\ldots,G_k$ are the connected components of $G$,
\[\lambda_{chr}(G)=\lambda_{chr}(G_1)\ldots \lambda_{chr}(G_k).\]
Then $\lambda_{chr}$ is an invertible element of $\mgr$, and we denote its inverse by $\lambda_0$.
Then, for any  graph $G$, $\lambda_0(G)=1$, or, equivalently:
\begin{align*}
&\forall G\in \gr,&\sum_{\sim\triangleleft G}\lambda_{chr}(G/\sim)&=
\sum_{\sim\triangleleft G}\lambda_{chr}(G\mid\sim)=\varepsilon'(G).
\end{align*}
Moreover, $P_{chr}=\phi_0\leftarrow \lambda_{chr}$, or equivalently:
\begin{align*}
&\forall G\in \gr,&P_{chr}(G)&=\sum_{\sim \triangleleft G} \lambda_{chr}(G|\sim)X^{cl(\sim)}.
\end{align*}
\end{cor}

\begin{proof} 
By Theorem \ref{theo8}, there exists a unique $\lambda\in \mgr$, such that $\phi_0=\phi_1\leftarrow \lambda$.
Then:
\begin{align*}
\varepsilon' \circ \phi_0&=\varepsilon'\circ (\phi_1\otimes \lambda)\circ \delta
=((\varepsilon' \circ\phi_1)\otimes  \lambda)\circ \delta
=(\varepsilon' \otimes \lambda)\circ \delta
=\varepsilon' * \lambda
=\lambda.
\end{align*}
Therefore, for any graph $G$, $\lambda(G)=\varepsilon'(X^{|V(G)|})=1$. 
As $\lambda(\grun)=1$, by Theorem \ref{theo8}, $\lambda$ is invertible, and then $\phi_1=\phi_0\leftarrow \lambda^{*-1}$.
For any graph $G$, by definition of $\delta$:
\[\phi_1(G)=\sum_{\sim \triangleleft G} \lambda^{*-1}(G|\sim)X^{cl(\sim)}.\]
If $G$ is connected, there exists a unique $\sim_1\triangleleft G$ such that $cl(\sim_1)=1$: 
this is the equivalence relation such that for any $x,y\in V(G)$,
$x\sim_1 y$. Moreover, $G\mid \sim_1=G$. 
Hence, the coefficient of $X$ in $P_{chr}(X)$ is $\lambda^{*-1}(G\mid \sim_1)+0=\lambda^{*-1}(G)$, so:
\[\lambda^{*-1}(G)=\frac{dP_{chr}(G)}{dX}(0)=\lambda_{chr}(G).\]
Consequently, $\lambda_{chr}^{-1}=\lambda$. \end{proof}

The character $\lambda_{chr}$ will be called the \emph{chromatic character}. Its inverse is denoted by $\lambda_0$. 
We extend it to any $\hgrd$  by $\lambda_{chr}^{(\calD)}=\lambda_{chr}\circ \calF^{(\calD)}$. 
Then its inverse is $\lambda_0^{(\calD)}==\lambda_0\circ \calF^{(\calD)}$. 
Then, as $\calF^{(\calD)}$ is compatible with both bialgebraic structures on $\hgrd$:
\begin{align*}
\phi_1^{(\calD)}&=\phi_0^{(\calD)}\leftarrow \lambda_{chr}^{(\calD)}.
\end{align*}

\begin{prop} \label{prop13}
$\lambda_{chr}(\grun)=1$; if $G$ is a connected graph, $G\neq \grun$, then:
\[\lambda_{chr}(G)=\sum_{\calf\in \F(G)}(-1)^{\sharp\calf}.\]
\end{prop}

\begin{proof} We have $\lambda_{chr}(\grun)=\lambda_0(\grun)=1$, so both $\lambda_{chr}$ and 
$\lambda_0$ can be seen as characters on $\hgr'$. Hence, for any connected graph $G$, different from $\grun$:
\begin{align*}
\lambda_{chr}(G)&=\lambda_0\circ S'(G)=\sum_{\calf\in \F(G)} (-1)^{\sharp \calf}\lambda_0(G_\calf)
=\sum_{\calf\in \F(G)} (-1)^{\sharp \calf},
\end{align*}
as $\lambda_0(H)=1$ for any graph $H\in \gr$. \end{proof}

\begin{example}\label{ex2.1}
\begin{enumerate}
\item By direct computations, we obtain:
\[\begin{array}{c|c|c|c|c|c|c|c|c|c|c}
G&\grun&\grdeux&\grtroisun&\grtroisdeux&\grquatreun&\grquatredeux&\grquatretrois&\grquatrequatre&\grquatrecinq&\grquatresix\\
\hline \lambda_{chr}(G)&1&-1&2&1&-6&-4&-2&-3&-1&-1
\end{array}\]
\item If $G$ is a complete graph with $n$ vertices, $P_{chr}(G)(X)=X(X-1)\ldots (X-n+1)$, so $\lambda_{chr}(G)=(-1)^{n-1} (n-1)!$.
\item If $G$ is a tree with $n$ vertices, $P_{chr}(G)(X)=X(X-1)^{n-1}$, so $\lambda_{chr}(G)=(-1)^{n-1}$. 
\end{enumerate}\end{example}

\subsection{Extraction and contraction of edges}

\begin{defi}
Let $G$ be a graph and $e\in E(G)$.
\begin{enumerate}
\item (Contraction of $e$). The graph $G/e$ is $G/\sim_e$, where $\sim_e$ is the equivalence which classes are $e$ and singletons.
\item (Subtraction of $e$). The graph $G\setminus e$ is the graph $(V(G),E(G)\setminus\{e\})$.
\item We shall say that $e$ is a \emph{bridge} (or an \emph{isthmus}) of $G$ if $cc(G\setminus e)>cc(G)$.
\end{enumerate}\end{defi}

We now give an algebraic proof of the following well-known result \cite{Harary}, which allows to compute the chromatic
polynomial by induction on the number of edges:

\begin{prop} \label{prop15}
For any graph $G$, for any edge $e$ of $G$:
\begin{align*}
P_{chr}(G)&=P_{chr}(G\setminus e)-P_{chr}(G/e);&
\lambda_{chr}(G)&=\begin{cases}
-\lambda_{chr}(G/e)\mbox{ if $e$ is a bridge},\\
\lambda_{chr}(G\setminus e)-\lambda_{chr}(G/e)\mbox{ otherwise}.
\end{cases}\end{align*}\end{prop}

\begin{proof} Let $G$ be a graph, and $e\in E(G)$. Let us prove that for all $k\geq 1$, $P_{chr}(G)(k)=P_{chr}(G\setminus e)(k)-P_{chr}(G/e)(k)$.
We proceed by induction on $k$. If $k=1$, $P_{chr}(G)(1)=\varepsilon'(G)=0$.
If $G$ has only one edge, then $G\setminus e$ and $G/e$ are totally disconnected, and:
\[P_{chr}(G\setminus e)(1)-P_{chr}(G/e)(1)=1-1=0.\]
Otherwise, $G\setminus e$ and $G/e$ have edges, and:
\[P_{chr}(G\setminus e)(1)-P_{chr}(G/e)(1)=0-0=0.\]
Let us assume the result at rank $k$. Putting $e=\{x,y\}$:
\begin{align*}
&P_{chr}(G\setminus e)(k+1)-P_{chr}(G/e)(k+1)\\
&=\sum_{V(G)=I\sqcup J}P_{chr}((G\setminus e)_{\mid I})(k)P_{chr}((G\setminus e)_{\mid J})(1)\\
&-\sum_{\substack{V(G)=I\sqcup J,\\x,y\in I}} P_{chr}((G/e)_{\mid I})(k)P_{chr}((G/e)_{\mid J}(1)
-\sum_{\substack{V(G)=I\sqcup J,\\x,y\in J}} P_{chr}((G/e)_{\mid I})(k)P_{chr}((G/e)_{\mid J}(1)\\
&=\sum_{\substack{V(G)=I\sqcup J,\\x,y\in I}}P_{chr}((G\setminus e)_{\mid I})(k)P_{chr}((G\setminus e)_{\mid J})(1)+
\sum_{\substack{V(G)=I\sqcup J,\\x,y\in J}}P_{chr}((G\setminus e)_{\mid I})(k)P_{chr}((G\setminus e)_{\mid J})(1)\\
&-\sum_{\substack{V(G)=I\sqcup J,\\x,y\in I}} P_{chr}((G/e)_{\mid I})(k)P_{chr}((G/e)_{\mid J})(1)
-\sum_{\substack{V(G)=I\sqcup J,\\x,y\in J}} P_{chr}((G/e)_{\mid I})(k)P_{chr}((G/e)_{\mid J})(1)\\
&+\sum_{\substack{V(G)=I\sqcup J,\\(x,y)\in (I\times J)\cup (J\times I)}}P_{chr}((G\setminus e)_{\mid I})(k)P_{chr}((G\setminus e)_{\mid J})(1)\\
&=\sum_{\substack{V(G)=I\sqcup J,\\x,y\in I}}P_{chr}((G_{\mid I})\setminus e)(k)P_{chr}(G_{\mid J})(1)+
\sum_{\substack{V(G)=I\sqcup J,\\x,y\in J}}P_{chr}(G_{\mid I})(k)P_{chr}((G_{\mid J})\setminus e)(1)\\
&-\sum_{\substack{V(G)=I\sqcup J,\\x,y\in I}} P_{chr}((G_{\mid I})/e)(k)P_{chr}(G_{\mid J})(1)
-\sum_{\substack{V(G)=I\sqcup J,\\x,y\in J}} P_{chr}(G_{\mid I})(k)P_{chr}((G_{\mid J})/e)(1)\\
&+\sum_{\substack{V(G)=I\sqcup J,\\(x,y)\in (I\times J)\cup (J\times I)}}P_{chr}(G_{\mid I})(k)P_{chr}(G_{\mid J})(1)\\
&=\sum_{\substack{V(G)=I\sqcup J,\\x,y\in I}}P_{chr}(G_{\mid I})(k)P_{chr}(G_{\mid J})(1)+
\sum_{\substack{V(G)=I\sqcup J,\\x,y\in J}}P_{chr}(G_{\mid I})(k)P_{chr}(G_{\mid J})(1)\\
&+\sum_{\substack{V(G)=I\sqcup J,\\(x,y)\in (I\times J)\cup (J\times I)}}P_{chr}(G_{\mid I})(k)P_{chr}(G_{\mid J})(1)\\
&=\sum_{V(G)=I\sqcup J}P_{chr}(G_{\mid I})(k)P_{chr}(G_{\mid J})(1)\\
&=P_{chr}(G)(k+1).
\end{align*}
So the result holds for all $k\geq 1$. Hence, $P_{chr}(G)=P_{chr}(G\setminus e)-P_{chr}(G/e)$. \\

Let us assume that $G$ is connected. Note that $G/e$ is connected. 
If $e$ is a bridge, then $G\setminus e$ is not connected; each of its connected components
belongs to the augmentation ideal of $\hgr$, so their images belong to the augmentation ideal of $\Q[X]$, that is to say $X\Q[X]$;
hence, $P_{chr}(G\setminus e)\in X^2\Q[X]$, so:
\begin{align*}
\lambda_{chr}(G)&=\frac{dP_{chr}(G)}{dX}(0)
=\frac{dP_{chr}(G\setminus e)}{dX}(0)- \frac{dP_{chr}(G/e)}{dX}(0)=0-\lambda_{chr}(G/e).
\end{align*}
Otherwise, $G\setminus e$ is connected, and:
\begin{align*}
\lambda_{chr}(G)&=\frac{dP_{chr}(G)}{dX}(0)
=\frac{dP_{chr}(G\setminus e)}{dX}(0)- \frac{dP_{chr}(G/e)}{dX}(0)=\lambda_{chr}(G\setminus e)-\lambda_{chr}(G/e).
\end{align*}

If $G$ is not connected, we can write $G=G_1G_2$, where $G_1$ is connected and $e$ is an edge of $G_1$. Then:
\begin{align*}
\lambda_{chr}(G)&=\lambda_{chr}(G_1)\lambda_{chr}(G_2)\\
&=\begin{cases}
-\lambda_{chr}(G_1/e)\lambda_{chr}(G_2)\mbox{ if $e$ is a bridge},\\
\lambda_{chr}(G_1\setminus e)\lambda_{chr}(G_2)-\lambda_{chr}(G_1/e)\lambda_{chr}(G_2)\mbox{ otherwise};
\end{cases}\\
&=\begin{cases}
-\lambda_{chr}((G_1/e)G_2)\mbox{ if $e$ is a bridge},\\
\lambda_{chr}((G_1\setminus e)G_2)-\lambda_{chr}((G_1/e)G_2)\mbox{ otherwise};
\end{cases}\\
&=\begin{cases}
-\lambda_{chr}(G/e)\mbox{ if $e$ is a bridge},\\
\lambda_{chr}(G\setminus e)-\lambda_{chr}(G/e)\mbox{ otherwise}.
\end{cases}
\end{align*}
So the result holds for any graph $G$.  \end{proof}

\begin{example}
For any $n\geqslant 3$, let us denote by $C_n$ the cyclic graph with $n$ vertices.
Then $\lambda_{chr}(C_3)=2$. Choosing any edge $e$ of $C_n$ with $n\geqslant 4$, 
$C_n/e=C_{n-1}$ and $C_n\setminus e$ is a chain on $n$ vertices, so is a tree. Hence:
\[\lambda_{chr}(C_n)=(-1)^{n-1}-\lambda_{chr}(C_{n-1}).\]
A direct induction proves that for any $n\geqslant 3$, $\lambda_{chr}(C_n)=(-1)^{n-1}(n-1)$. 
\end{example}

\subsection{Lattices attached to graphs}

We here make the link with Rota's methods for proving the alternation of signs in the coefficients of chromatic polynomials.\\

The following order is used to prove Proposition \ref{prop4}:

\begin{prop} \label{prop16}
Let $G$ be a graph. We denote by $\rel(G)$ the set of equivalences $\sim$ on $V(G)$, such that $\sim\triangleleft G$.
Then $\rel(G)$ is partially ordered by refinement: 
\[\forall \sim, \sim' \in\rel(G),\:\sim\leq \sim' \mbox{ if } (\forall x,y\in V(G),\: x\sim y \Longrightarrow x\sim' y).\]
In other words, $\sim \leq \sim'$ if the equivalence classes of $\sim'$ are disjoint unions of equivalence classes of $\sim$.
Then $(\rel(G),\leq)$ is a bounded graded lattice.
Its minimal element $\sim_0$ is the equality; its maximal element $\sim_1$ is the relation which equivalence classes are the 
connected components of $\rel(G)$.
\end{prop}

\begin{proof} Let $\sim,\sim'\in \rel(G)$. We define $\sim \wedge \sim'$ as the equivalence which classes are the connected components
of the subsets $Cl_\sim(x)\cap Cl_{\sim'}(y)$, $x,y\in V(G)$. By its very definition, $\sim \wedge \sim' \triangleleft G$,
and $\sim\wedge \sim'\leq \sim,\sim'$. If $\sim'' \leq \sim,\sim'\leq$ in $\rel(G)$, then the equivalence classes of $\sim$ and $\sim'$
are disjoint union of equivalence classes of $\sim''$, so their intersections also are; as the equivalence classes of $\sim''$ are connected,
the connected components of these intersections are also disjoint union of equivalence classes of $\sim''$. This means that 
$\sim'' \leq \sim\wedge \sim'$.

We define $\sim \vee \sim'$ as the relation defined on $V(G)$ in the following way:
for all $x,y\in V(G)$, $x\sim \vee \sim' y$ if there exists $x_1,x'_1,\ldots,x_k,x'_k \in V(G)$ such that:
\[x=x_1\sim x'_1 \sim' x_2\sim \ldots \sim' x_k \sim' x'_k=y.\]
It is not difficult to prove that $\sim\vee \sim'$ is an equivalence. 
Moreover, if $x\sim y$, then $x\sim \vee \sim' y$ ($x_1=x$, $x'_1=y$);
if $x\sim' y$, then $x\sim \vee \sim' y$ ($x_1=x'_1=x$, $x_2=x'_2=y$). 
Let $C$ be an equivalence class of $\sim \vee \sim'$, and let $x,y\in C$. With the preceding notations,
as the equivalence classes of $\sim$ and $\sim'$ are connected, for all $p\in [k]$, there exists a path from $x_p$ to $x'_p$,
formed of elements $\sim$-equivalent, hence $\sim\vee \sim'$-equivalent; for all $p\in [k-1]$, there exists a path from $x'_p$ to $x'_{p+1}$,
formed of elements $\sim'$-equivalent, hence $\sim\vee \sim'$-equivalent. Concatening these paths, we obtain a path from $x$ to $y$ in $C$,
which is connected. So $\sim \vee \sim' \in \rel(G)$, and $\sim,\sim' \leq \sim\vee \sim'$. 
Moreover, if $\sim,\sim' \leq \sim''$, then obviously $\sim \vee \sim'\leq \sim''$. We proved that $\rel(G)$ is a lattice.\\

For any $\sim \in \rel(G)$, we put $\deg(G)=|G|-cl(\sim)$. Note that $\deg(\sim_0)=0$. Let us assume that $\sim$ is covered by $\sim'$ in $\rel(G)$.
We denote by $C_1,\ldots, C_k$ the classes of $\sim$. As $\sim \leq \sim'$, the classes of $\sim'$ are disjoint unions of $C_p$;
as $\sim \neq \sim'$, one of them, denoted by $C'$, contains at least two $C_p$. As $C'$ is connected, there is an edge in $C'$
connecting two different $C_p$; up to a reindexation, we assume that there exists an edge from $C_1$ to $C_2$ in $C'$.
Then $C_1\sqcup C_2$ is connected, and the equivalence $\sim''$ which classes are $C_1\sqcup C_2,C_3,\ldots C_k$
satisfies $\sim \leq \sim'' \leq \sim'$. As $\sim'$ covers $\sim$, $\sim'=\sim''$, so $\deg(\sim')=|G|-k+1=\deg(\sim)+1$. \end{proof}

\begin{remark}
 This lattice is isomorphic to the one of \cite{Rota}. The isomorphism between them
sends a element $\sim \in \rel(G)$ to the partition formed by its equivalence classes.
\end{remark}

\begin{example}
 We represent $\sim \in \rel(G)$ by $G|\sim$. Here are examples of $\rel(G)$, represented by their Hasse graphs.
We index the vertices of the graphs by letters for a better understanding.
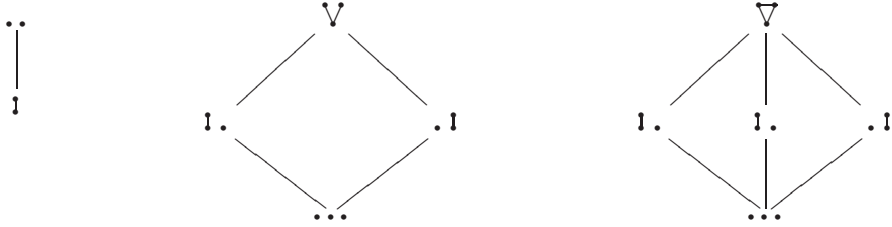
\begin{align*}
&\xymatrix{\grddeux{$a$}{$b$}\ar@{-}[d]\\ \grdun{$a$}\grdun{$b$} }&
&\xymatrix{&\grdtroisdeux{$a$}{$c$}{$b$} \ar@{-}[ld] \ar@{-}[rd]&\\
\grddeux{$a$}{$b$}\grdun{$c$}\ar@{-}[rd]&&\grddeux{$a$}{$c$}\grdun{$b$}\ar@{-}[ld]\\
&\grdun{$a$}\grdun{$b$}\grdun{$c$}&}&
&\xymatrix{&\grdtroisun{$a$}{$c$}{$b$} \ar@{-}[ld] \ar@{-}[d] \ar@{-}[rd]&\\
\grddeux{$b$}{$c$}\grdun{$a$}\ar@{-}[rd]&\grddeux{$a$}{$c$}\grdun{$b$} \ar@{-}[d]
&\grddeux{$a$}{$b$}\grdun{$c$}\ar@{-}[ld]\\
&\grdun{$a$}\grdun{$b$}\grdun{$c$}&} 
\end{align*}\end{example}

\begin{prop}
Let $G$ be a graph. We denote by $\mu_G$ the Möbius function of $\rel(G)$.
\begin{enumerate}
\item If $\sim \leq \sim'$ in $\rel(G)$, then the poset $[\sim,\sim']$ is isomorphic to $\rel((G|\sim')/\sim)$.
\item For any $\sim\leq \sim'$ in $\rel(G)$, $\mu_G(\sim,\sim')=\lambda_{chr}((G|\sim')/\sim)$.
In particular:
\[\mu_G(\sim_0,\sim_1)=\lambda_{chr}(G).\]
\end{enumerate} \end{prop}

\begin{proof} Let $\sim \leq \sim' \in \rel(G)$. If $\sim''$ is an equivalence on $V(G)$, then $\sim \leq \sim'' \leq \sim$ if, and only if,
the following conditions are satisfied:
\begin{itemize}
\item $\sim''$ goes to the quotient $G/\sim$, as an equivalence denoted by $\overline{\sim''}$.
\item $\overline{\sim''} \in \rel((G|\sim')/\sim)$.
\end{itemize}
Hence, we obtain a map from $[\sim,\sim']$ to $\rel((G|\sim')/\sim)$, sending $\sim''$ to $\overline{\sim''}$. 
It is immediate that this is a lattice isomorphism.

Let $\sim \leq \sim'\in \rel(G)$. As $[\sim,\sim']$ is isomorphic to the lattice $\rel((G|\sim')/\sim)$:
\begin{align*}
\sum_{\sim\leq \sim'' \leq \sim'} \lambda_{chr}((G|\sim'')/\sim)
&=\sum_{\overline{\sim''}\in \rel((G|\sim')/\sim)} \lambda_{chr}(((G/\sim')/\sim)|\overline{\sim''})\\
&=P_{chr}((G|\sim')/\sim)(1)\\
&=\begin{cases}
1\mbox{ if $(G|\sim')/\sim$ is totally disconnected},\\
0\mbox{ otherwise};
\end{cases}\\
&=\begin{cases}
1\mbox{ if $\sim=\sim'$},\\
0\mbox{ otherwise}.
\end{cases}
\end{align*}
Hence, $\mu_G(\sim,\sim')=\lambda_{chr}((G|\sim')/\sim)$. \end{proof}

\begin{remark}
 We now use the notion of incidence algebra of a family of posets exposed in \cite{Schmitt}.
We consider the family of posets:
\[\{[\sim,\sim'] \mid G\in \gr,\sim\leq \sim' \mbox{ in }\rel(G)\}.\]
It is obviously interval-closed. We define an equivalence relation on this family as the one generated by
$[\sim,\sim'] \equiv \rel((G|\sim')/\sim)$. The incidence bialgebra associated to this family is $(\hgr,m,\delta)$.
\end{remark}

\begin{prop} \label{prop18}
Let $G$ a graph. 
\begin{enumerate}
\item Let $G_1,\ldots,G_k$ be the connected components of $G$. Then $\rel(G)\approx \rel(G_1)\times \ldots \times \rel(G_k)$.
\item Let $e$ be a bridge of $G$. Then $\rel(G)\approx \rel(G/e)\times \rel(\grdeux)$.
\item We consider the following map:
\[\zeta_G:\left\{\begin{array}{rcl}
\rel(G)&\longrightarrow&\mathcal{P}(E(G))\\
\sim&\longrightarrow&E(G|\sim).
\end{array}\right.\]
This map is injective; for any $\sim$, $\sim'\in \rel(G)$, $\sim \leq \sim'$ if, and only if, $\zeta_G(\sim)\subseteq \zeta_G(\sim')$.
Moreover, $\zeta_G$ is bijective if, and only if, $G$ is a forest -- that is to say a graph such that any edge is a bridge.
\end{enumerate} \end{prop}

\begin{proof} 1. If $G,H$ are graphs and $\sim$ is an equivalence on $V(GH)$, then $\sim \triangleleft GH$ if, and only if:
\begin{itemize}
\item $\sim_{\mid V(G)} \triangleleft G$.
\item $\sim_{\mid V(H)} \triangleleft H$.
\item For any $x,y\in V(G)\sqcup V(H)$, $(x\sim y)\Longrightarrow ((x,y)\in V(G)^2\sqcup V(H)^2$.
\end{itemize}
Hence, the map sending $\sim$ to $(\sim_{\mid V(G)},\sim_{\mid V(H)})$ from $\rel(GH)$ to $\rel(G)\times \rel(H)$ is an isomorphism;
the first point follows.\\

2. Note that $\rel(\grdeux)=\{\grun\grun,\grdeux\}$, with $\grun\grun\leq \grdeux$.
By the first point, it is enough to prove it if $G$ is connected. Let us put $e=\{x',x''\}$, $G'$, respectively $G''$,
the connected components of $G\setminus e$ containing $x'$, respectively $x''$. We define a map $\psi:\rel(G/e)\times \rel(\grdeux)$
to $\rel(G)$ in the following way: if $\overline{\sim} \triangleleft \rel(G/e)$,
\begin{itemize}
\item $\psi(\overline{\sim},\grdeux)=\sim$, defined by $x\sim y$ if $\overline{x} \overline{\sim} \overline{y}$.
This is clearly an equivalence; moreover, $x'\sim x''$. if $x\sim y$, there exists a path from $\overline{x}$ to $\overline{y}$ in $G/e$,
formed by vertices $\overline{\sim}$-equivalent to $\overline{x}$ and $\overline{y}$. Adding edges $e$ if needed in this path,
we obtain a path from $x$ to $y$ in $G$, formed by vertices $\sim$-equivalent to $x$ and $y$; hence, $\sim \triangleleft G$.
\item $\psi(\overline{\sim},\grun\grun)=\sim$, defined by $x\sim y$ if $\overline{x}\overline{\sim} \overline{y}$ 
and $(x,y)\in V(G')^2\sqcup V(G'')^2$. This is clearly an equivalence; moreover, we do not have $x'\sim x''$. If $x\sim y$, let us assume for example
that both of them belong to $G'$. There is a path in $G\setminus e$ from $\overline{x}$ to $\overline{y}$, formed by vertices
formed by vertices $\overline{\sim}$-equivalent to $\overline{x}$ and $\overline{y}$. We choose such a path of minimal length.
If this path contains vertices belonging to $G''$, as $e$ is a bridge of $G$, it has the form:
\[\overline{x}-\ldots -\overline{x'}-\ldots -\overline{x'}-\ldots-\overline{y}.\]
Hence, we can obtain a shorter path from $\overline{x}$ to $\overline{y}$: this is a contradiction. So all the vertices of this path belong
to $G'$; hence, they are all $\sim$-equivalent. Finally, $\sim \triangleleft G$.
\end{itemize}

Let us assume that $\psi(\overline{\sim},\grdeux)=\psi(\overline{\sim}',\grdeux)=\sim$. If $\overline{x}\overline{\sim} \overline{y}$,
then $x\sim y$, so $\overline{x}\overline{\sim} \overline{y}$; by symmetry, $\overline{\sim}=\overline{\sim}'$.
Let us assume that $\psi(\overline{\sim},\grun\grun)=\psi(\overline{\sim}',\grun\grun)=\sim$. If $\overline{x}\overline{\sim} \overline{y}$:
\begin{itemize}
\item If $x,y\in V(G')$ or $x,y\in V(G'')$, then $x\sim y$, so $\overline{x}\overline{\sim} \overline{y}$.
\item If $(x,y)\in V(G')\times V(G'')$ or $(x,y)\in V(G'')\times V(G')$, up to a permutation we can assume that $x\in V(G')$ and $y\in V(G'')$.
As $\overline{\sim} \triangleleft G/e$, there exists a path from $\overline{x}$ to $\overline{y}$ formed by $\overline{\sim}$-equivalent vertices.
This path necessarily goes via $\overline{x'}=\overline{x''}$. Hence, $x\sim x'$ and $y\sim x''$, so $\overline{x}\overline{\sim}\overline{x'}$
and $\overline{y}\overline{\sim}\overline{x''}$, and finally $\overline{x}\overline{\sim} \overline{y}$.
\end{itemize}
By symmetry, $\overline{\sim}=\overline{\sim}'$. We proved that $\psi$ is injective.

Let $\sim \triangleleft G$. If $x'\sim x''$, then $\sim$ goes through the quotient $G/e$ and gives an equivalence $\overline{\sim}\triangleleft G/e$.
Moreover, $\psi(\overline{\sim},\grdeux)=\sim$. Otherwise, $\sim \triangleleft G\setminus e=G'G''$;
let us denote the equivalence classes of $\sim$ by $C_1,\ldots,C_{k+l}$, with 
$x'\in C_1$, $x''\in C_{k+1}$, $C_1,\ldots,C_k\subseteq V(G')$, $C_{k+1},\ldots,C_{k+l}\subseteq V(G'')$. 
Let $\overline{\sim}$ the equivalence on $V(G/e)$ which equivalence classes are
$\overline{C_1\sqcup C_{k+1}}, \overline{C_2},\ldots,\overline{C_k},\overline{C_{k+2}},\ldots,\overline{C_{k+l}}$.
Then $\overline{\sim}\triangleleft G/e$ and $\psi(\overline{\sim},\grun\grun)=\sim$. We proved that $\psi$ is surjective.

It is immediate that $\psi(\overline{\sim}_1,\sim_2)\leq  \psi(\overline{\sim}'_1,\sim'_2)$ if, and only if,
$\overline{\sim}_1\leq \overline{\sim}_1'$ and $\sim_2\leq \sim'_2$. So $\psi$ is a lattice isomorphism.\\

3. Let $\sim$, $\sim'$ be elements of $\rel(G)$. If $\sim \leq \sim'$, then the connected components of $G|\sim'$ are disjoint unions
of connected components of $G|\sim$, so $E(G|\sim)\subseteq E(G|\sim')$. 

If $E(G|\sim) \subseteq E(G|\sim')$, then the connected components
of $G|\sim'$ are disjoint unions of connected components of $G|\sim$, so $\sim \leq \sim'$.

Consequently, if $\zeta_G(\sim)=\zeta_G(\sim')$, then $\sim \leq \sim'$ and $\sim'\leq \sim$, so $\sim=\sim'$: $\zeta_G$ is injective.\\

Let us assume that $\zeta_G$ is surjective. Let $e\in E(G)$; we consider $\sim\in \rel(G)$, such that $\zeta_G(\sim)=E(G)\setminus e$.
In other words, $G|\sim=G\setminus e$. Hence, $\sim \neq \sim_1$, so $cl(\sim)<cl(\sim_1)$: $G|\sim$ has strictly more connected components
than $G$. This proves that $e$ is a bridge, so $G$ is a forest.\\

Let us assume that $G$ is a forest. We denote by $k$ the number of its edges. As any edge of $G$ is a bridge,
by the second point, $\rel(G)$ is isomorphic to $\rel(\grdeux)^k \times \rel(\grun)^{cc(G)}$, so is of cardinal $2^k\times 1^{cc(G)}=2^k$.
Hence, $\zeta_G$ is surjective. \end{proof}

\begin{remark}
As a consequence, isomorphic posets may be associated to non-isomorphic graphs: for example,
$\rel(\grquatrecinq)\approx \rel(\grquatresix)\approx \rel(\grdeux)^3$.
\end{remark}

\subsection{Applications}

\begin{cor} \label{cor19}
Let $G$ be a graph. 
\begin{enumerate}
\item  $\lambda_{chr}(G)$ is non-zero, of sign $(-1)^{\deg(G)}$.
\item We put $P_{chr}(G)=a_0+\ldots+a_nX^n$.
\begin{itemize}
\item For any $i$, $a_i \neq 0$ if, and only if, $cc(G)\leq i\leq |G|$.
\item If $cc(G)\leq i\leq |G|$, the sign of $a_i$ is $(-1)^{|G|-i}$.
\end{itemize}
\item $-a_{|G|-1}$ is the number of edges of $|G|$.
\end{enumerate}\end{cor}

\begin{proof} 1. For any graph $G$, we put $\tilde{\lambda}_{chr}(G)=(-1)^{\deg(G)}\lambda_{chr}(G)$. This defines 
a character $\tilde{\lambda}\in \mgr$.
Let us prove that for any edge $e$ of $G$:
\begin{align*}
\tilde{\lambda}_{chr}(G)&=\begin{cases}
\tilde{\lambda}_{chr}(G/e)\mbox{ if $e$ is a bridge},\\
\tilde{\lambda}_{chr}(G\setminus e)+\tilde{\lambda}_{chr}(G/e)\mbox{ otherwise}.
\end{cases}\end{align*}
We proceed by induction on the number $k$ of edges of $G$.
If $k=0$, there is nothing to prove. Let us assume the result at all ranks $<k$,
with $k\geq 1$. Let $e$ be an edge of $G$. We shall apply the induction hypothesis to $G/e$ and $G\setminus e$.
Note that $cc(G/e)=cc(G)$ and $|G/e|=|G|-1$, so $\deg(G/e)=\deg(G)-1$.
\begin{itemize}
\item If $e$ is a bridge, then:
\[\lambda_{chr}(G)=-(-1)^{\deg(G/e)} \tilde{\lambda}_{chr}(G/e)=(-1)^{\deg(G)} \tilde{\lambda}_{chr}(G/e).\]
\item If $e$ is not a bridge, then $cc(G\setminus e)=cc(G)$, and $|G\setminus e|=|G|$, so $\deg(G\setminus e)=\deg(G)$. Hence:
\begin{align*}
\lambda_{chr}(G/e)&=(-1)^{\deg(G\setminus e)}\tilde{\lambda}_{chr}(G\setminus e)-(-1)^{\deg(G/e)}\tilde{\lambda}_{chr}(G/e)\\
&=(-1)^{\deg(G)}\tilde{\lambda}_{chr}(G\setminus e)+(-1)^{\deg(G)}\tilde{\lambda}_{chr}(G/e)\\
&=(-1)^{\deg(G)}(\tilde{\lambda}_{chr}(G\setminus e)+\tilde{\lambda}_{chr}(G/e)).
\end{align*}\end{itemize}
So the result holds for any graph $G$. \\

If $G$ has no edge, then $\deg(G)=0$ and $\lambda_{chr}(G)=\tilde{\lambda}_{chr}(G)=1$.
An easy induction on the number of edges proves that for any graph $G$, $\tilde{\lambda}_{chr}(G)\geq 1$. \\

2. By Corollary \ref{cor12}, for any $i$:
\begin{align*}
a_i&=\sum_{\sim\triangleleft G,\:cl(\sim)=i} \lambda_{chr}(G|\sim)\\
&=\sum_{\sim\triangleleft G,\:cl(\sim)=i}(-1)^{|G|-i}\tilde{\lambda}_{chr}(G|\sim)\\
&=(-1)^{|G|-i}\sum_{\sim\triangleleft G,\:cl(\sim)=i}\tilde{\lambda}_{chr}(G|\sim).
\end{align*}
As for any graph $H$, $\tilde{\lambda}_{chr}(H)\geqslant 1$, this is non-zero if, and only if, 
there exists a relation $\sim\triangleleft G$, such that $cl(\sim)=i$.
If this holds, the sign of $a_i$ is $(-1)^{|G|-i}$. It remains to prove that there  exists a relation $\sim\triangleleft G$, 
such that $cl(\sim)=i$ if, and only if, $cc(G)\leq i \leq |G|$.\\

$\Longrightarrow$. If $\sim \triangleleft G$, with $cl(\sim)=i$, as the equivalence classes of $\sim$ are connected, each connected component of $G$
is a union of classes of $\sim$, so $i\geq cc(G)$. Obviously, $i\leq |G|$.

$\Longleftarrow$. We proceed by decreasing induction on $i$. If $i=|G|$, then the equality of $V(G)$ answers the question.
Let us assume that $cc(G)\leq i<|G|$ and that the result holds at rank $i+1$. Let $\sim' \triangleleft G$, with $cl(\sim')=i+1$.
We denote by $I_1,\ldots,I_{i+1}$ the equivalence classes of $\sim'$. As $I_1,\ldots, I_{i+1}$ are connected, the connected components of $G$
are union of $I_p$; as $i+1>cc(G)$, one of the connected components of $G$, which we call $G'$, contains at least two equivalence classes of $\sim'$.
As $G'$ is connected, there exists an edge in $G'$, relation two vertices into different equivalence classes of $\sim'$; up to a reindexation,
we assume that they are $I_1$ and $I_2$. Hence, $I_1\sqcup I_2$ is connected.
We consider the relation $\sim$ which equivalence classes are $I_1\sqcup I_2,I_3,\ldots,I_{i+1}$: then $\sim \triangleleft G$ and $cl(\sim)=i$. \\

3. For $i=|G|-1$, we have to consider relations $\sim\triangleleft G$ such that $cl(\sim)=|G|-1$. These equivalences are in bijection with edges, via the map
$\zeta_G$ of Proposition \ref{prop18}.
For such an equivalence, $G|\sim=\grdeux \grun^{|G|-1}$, so $\lambda_{chr}(G|\sim)=-1$. Finally, $a_i=-|E(V)|$. \end{proof}

\begin{remark}
 The result on the signs of the coefficients of $P_{chr}(G)$ is due to Rota \cite{Rota}, who proved it using 
the Möbius function of the poset of Proposition \ref{prop18}.
\end{remark}

\begin{cor}
Let $G$ be a graph; $|\lambda_{chr}(G)|=1$ if, and only if, $G$ is a forest.
\end{cor}

\begin{proof}
$\Longleftarrow$. Then each component of $G$ is a tree. The result then comes from Example \ref{ex2.1}, last point.\\

$\Longrightarrow$. If $G$ is not a forest, there exists an edge $e$ of $G$ which is not a bridge. Then:
\[|\lambda_{chr}(G)|=|\lambda_{chr}(G\setminus e)|+|\lambda_{chr}(G/e)|\geq 1+1=2.\]
So $|\lambda_{chr}(G)|\neq 1$. \end{proof}

\begin{lemma}
If $G$ is a graph and $e$ is a bridge of $G$, then:
\[\lambda_{chr}(G)=-\lambda_{chr}(G\setminus e)=-\lambda_{chr}(G/e).\]
\end{lemma}

\begin{proof} We already proved in Proposition \ref{prop18} that $\lambda_{chr}(G)=-\lambda_{chr}(G/e)$. 
Let us prove that $\lambda_{chr}(G)=-\lambda_{chr}(G\setminus e)$ by induction on the number $k$ of edges of $G$ which are not bridges.
If $k=0$, then $G$ and $G\setminus e$ are forests with $n$ vertices,
$cc(G\setminus e)=cc(G)+1$ and:
\[\lambda_{chr}(G)=-\lambda_{chr}(G\setminus e)=(-1)^{\deg(G)}.\]
Let us assume the result at rank $k-1$, $k\geq 1$. Let $f$ be an edge of $G$ which is not a bridge of $G$.
\begin{align*}
\lambda_{chr}(G)&=\lambda_{chr}(G\setminus f)-\lambda_{chr}(G/f)\\
&=-\lambda_{chr}((G\setminus f)\setminus e)+\lambda_{chr}((G/f)\setminus e)\\
&=-\lambda_{chr}((G\setminus e)\setminus f)+\lambda_{chr}((G\setminus e)/f)\\
&=-\lambda_{chr}(G\setminus e).
\end{align*}
So the result holds for any bridge of any graph. \end{proof}

\begin{prop}\begin{enumerate}
\item Let $G$ and $H$ be two graphs, with $V(G)=V(H)$ and $E(G)\subseteq E(H)$. Then:
\[|\lambda_{chr}(G)|\leq |\lambda_{chr}(H)|+cc(G)-cc(H)-\sharp(E(H)-E(G))\leq |\lambda_{chr}(H)|.\]
Moreover, if $cc(G)=cc(H)$, then $|\lambda_{chr}(G)|=|\lambda_{chr}(H)|$ if, and only if, $G=H$.
\item For any graph $G$, $|\lambda_{chr}(G)|\leq (|G|-1)!$, with equality if, and only if, $G$ is complete.
\end{enumerate}\end{prop}

\begin{proof} 1. We put $k=\sharp(E(H)\setminus E(G))$. There exists a sequence $e_1,\ldots,e_k$ of edges of $H$ such that:
\begin{align*}
G_0&=G,&G_k&=H,&\forall i\in [k], \: G_{i-1}&=G_i\setminus e_i.
\end{align*}
For all $i$, $cc(G_i)=cc(G_{i-1})+1$ if $e_i$ is a bridge of $G_i$, and $cc(G_i)=cc(G_{i-1})$ otherwise.
Hence, $cc(G)-cc(H)\leq k$. We denote by $I$ the set of indices $i$ such that $cc(G_i)=cc(G_{i-1})$; then $\sharp I=k-cc(G)+cc(H)$. Moreover:
\[|\lambda_{chr}(G_i)|=\begin{cases}
|\lambda_{chr}(G_{i-1})|+|\lambda_{chr}((G_i)/e_i)|>|\lambda_{chr}(G_{i-1})| \mbox{ if $i\in I$},\\
|\lambda_{chr}(G_{i-1})| \mbox{ if $i\notin I$}.
\end{cases}\]
As a conclusion, $|\lambda_{chr}(G)|\leq |\lambda_{chr}(H)|-\sharp I=|\lambda_{chr}(H)|+cc(G)-cc(H)-k\leq |\lambda_{chr}(H)|$.\\

If $cc(G)=cc(H)$ and $|\lambda_{chr}(G)|=|\lambda_{chr}(H)|$, then $k=0$, so $G=H$. \\

2. We put $n=|G|$. We apply the first point with $H$ the complete graph such that $V(H)=V(G)$. We already observed that $|\lambda_{chr}(H)|=(n-1)!$, so:
\[|\lambda_{chr}(G)| \leq (n-1)!.\]
If $G$ is not connected, there exist graphs $G_1$, $G_2$ such that $G=G_1G_2$, $n_1=|G_1]<n$, $n_2=|G_2|<n$. Hence:
\[|\lambda_{chr}(G)|=|\lambda_{chr}(G_1)||\lambda_{chr}(G_2)\leq (n_1-1)!(n_2-1)!\leq (n_1+n_2-2)<(n-1)!.\]
If $G$ is connected, then $cc(G)=cc(H)$: if $|\lambda_{chr}(G)|=|\lambda_{chr}(H)|$, then $G=H$. \end{proof}

\subsection{Values of the chromatic polynomial at negative integers}

\begin{theo} \label{theo23}
Let $k\geq 1$ and $G$ a graph. Then $(-1)^{|G|}P_{chr}(G)(-k)$ is the number of families $((I_1,\ldots,I_k),O_1,\ldots,O_k)$ such that:
\begin{itemize}
\item $I_1\sqcup \ldots \sqcup I_k=V(G)$ (note that one may have empty $I_p$'s).
\item For all $1\leq i \leq k$, $O_i$ is an acyclic orientation of $G_{\mid I_i}$.
\end{itemize}
In particular, $(-1)^{|G|} P_{chr}(G)(-1)$ is the number of acyclic orientations of $G$.
\end{theo}

\begin{proof} By the extraction-contraction process:
\begin{itemize}
\item If $G$ is totally disconnected, $(-1)^{|G|} P_{chr}(G)(-1)=1$.
\item If $G$ has an edge $e$, $(-1)^{|G|}P_{chr}(G)(-1)=(-1)^{|G\setminus e|} P_{chr}(G\setminus e)(-1)
+(-1)^{|G/e|} P_{chr}(G/e)(-1)$.
\end{itemize}
For any graph $H$, let us denote by $\calA(H)$ the set of acyclic orientations of $H$. 
Let $G$ be a graph and $e=\{x,y\}$ be an edge of $G$. 
If $\sigma \in \calA(G/e)$,
we deduce an orientation $\overline{\sigma}$ of $G\setminus e$ by lifting the orientations of the edges of $G/e$ to
the edges of $G\setminus e$. Obviously, this defines an injective map $\iota$ from $\calA(G/e)$ to $\calA(G\setminus e)$.

If $\sigma\in \calA(G/e)$, let us denote by $\iota_+(\sigma)$, respectively $\iota_-(G)$, 
 the orientation of $G$ obtained from $\iota(\sigma)$ by orientating $e$ from $x$ to $y$, respectively from $y$ to $x$.
Let us assume that one of them is not acyclic. We obtain for example a cycle 
 \[x\rightarrow y\rightarrow x_1\rightarrow\ldots \rightarrow x_k=x,\]
 which induces a cycle in the orientation $\sigma$ of $G/e$:  this is a contradiction. 
 We obtain two maps $\iota_+,\iota_-:\calA(G/e)\longrightarrow \calA(G)$,  both injective, with disjoint images. 
  
 Let $\sigma \in \calA(G\setminus e)\setminus \iota(\calA(G/e))$. We denote by $\sigma_+$,
 respectively $\sigma_-$, the orientation of $G$ obtained from $\sigma$ by orientating $e$ from $x$ to $y$, respectively from $y$ to $x$.
 As $\sigma \notin \iota(\calA(G/e))$, there exists a vertex $z\in V(G)$, with edges $\{x,z\}$ and $\{y,z\}$,
 such that, $\{x,z\}$ is oriented from $x$ to $z$ and $\{y,z\}$ from $z$ to $y$, up to a permutation of $x$ and $y$.
 Then $y\rightarrow x\rightarrow z \rightarrow y$ is a cycle in $\sigma_-$ : at most one of $\sigma_+$ and $\sigma_-$
 is acyclic. Let us assume that none of them is acyclic. We obtain two cycles in $\sigma_+$ and $\sigma_-$:
 \begin{align*}
 &x\rightarrow y\rightarrow y_1 \ldots \rightarrow y_k=x,&
 &y\rightarrow x\rightarrow x_1 \ldots \rightarrow x_l=y.
 \end{align*}
We obtain then a cycle $y\rightarrow y_1\ldots \rightarrow y_k\rightarrow x_1\rightarrow \ldots \rightarrow x_l=y$
in $\sigma$, which is not acyclic. Hence, exactly one of $\sigma_-$ and $\sigma_+$ is acyclic:
we obtain an injective map $\kappa:\calA(G\setminus e)\setminus \iota(\calA(G/e))\longrightarrow \calA(G)$.
Clearly, the images of three maps are disjoint and cover the whole $\calA(G)$. Hence:
\[|\calA(G)|=2|\calA(G/e)|+|\calA(G\setminus e)\setminus \calA(G/e)|=|\calA(G/e)|+|\calA(G\setminus e)|.\]
An easy induction on the number of edges of $G$ then proves that $(-1)^{|G|}P_{chr}(G)(-1)$ is indeed $|\calA(G)|$. \\

If $k\geq 2$:
\begin{align*}
(-1)^{|G|}P_{chr}(G)(-k)&=(-1)^{|G|}P_{chr}(G)((-1)+\ldots+(-1))\\
&=(-1)^{|G|}\Delta^{(k-1)}\circ P_{chr}(G)(-1,\ldots,-1)\\
&=(-1)^{|G|}P_{chr}^{\otimes k}\circ \Delta^{(k-1)}(G)(-1,\ldots,-1)\\
&=(-1)^{|G|}\sum_{V(G)=I_1\sqcup \ldots \sqcup I_k}P_{chr}(G_{\mid I_1})(-1)\ldots P_{chr}(G_{\mid I_k})(-1)\\
&=\sum_{V(G)=I_1\sqcup \ldots \sqcup I_k} (-1)^{|G_{\mid I_1}|}P_{chr}(G_{\mid I_1})(-1)\ldots 
(-1)^{|G_{\mid I_k}|}P_{chr}(G_{\mid I_k})(-1).
\end{align*}
The case $k=1$ implies the result. \end{proof}

We recover the interpretation of Stanley  \cite{StanleyAcyclic}:

\begin{cor} \label{cor24}
Let $k\geq 1$ and $G$ a graph. Then $(-1)^{|G|}P_{chr}(G)(-k)$ is the number of pairs $(f,O)$ where
\begin{itemize}
\item $f$ is a map from $V(G)$ to $[k]$.
\item $O$ is an acyclic orientation of $G$.
\item If there is an oriented edge from $x$ to $y$ in $V(G)$ for the orientation $O$, then $f(x)\leq f(y)$.
\end{itemize} \end{cor}

\begin{proof} Let $A$ be the set of families defined in Theorem \ref{theo23}, and $B$ be the set of pairs defined in Corollary \ref{cor24}.
We define a bijection $\theta:A\longrightarrow B$ in the following way: if $((I_1,\ldots,I_k),O_1,\ldots,O_k) \in A$, 
we put $\theta((I_1,\ldots,I_k),O_1,\ldots,O_k)=(f,O)$, such that:
\begin{enumerate}
\item $f^{-1}(p)=I_p$ for any $p\in [k]$.
\item If $e=\{x,y\} \in E(G)$, we put $f(x)=i$ and $f(x)=j$. If $i=j$, then $e$ is oriented as in $O_i$. Otherwise, if $i<j$,
$e$ is oriented from $i$ to $j$ if $i<j$ and from $j$ to $i$ if $i>j$.
\end{enumerate}
Note that $O$ is indeed acyclic: if there is an oriented path from $x$ to $y$ in $G$ of length $\geq 1$,
then $f$ increases along this path. If $f$ remains constant, as $O_{f(x)}$ is acyclic, $x\neq y$. Otherwise, $f(x)<f(y)$, so $x\neq y$. 
It is then not difficult to see that $\theta$ is bijective. \end{proof}

This gives us a formula for the antipode of $(\hgrd,m,\Delta)$, proved in \cite{HumpertMartin} in another way
in the nondecorated case:

\begin{cor}\label{corantipode}
Let us denote by $S$ the antipode of $(\hgrd,m,\Delta)$. For any graph $G\in \gr(\calD)$:
\[S(G)=\sum_{\sim \triangleleft G} (-1)^{cl(\sim)} \sharp\{\mbox{acyclic orientations of $G/\sim$}\} G\mid \sim.\]
\end{cor}

\begin{proof}
Let us denote by $\star$ the convolution product associated to $\Delta$ in $\mgrd$, and by $\mu=\varepsilon'\circ S$ 
the inverse of $\varepsilon'$ for $\star$. 
Let us put $T=(\mu\otimes \id)\circ \delta$. Then, in the convolution algebra $\mathrm{End}(\hgrd)$, with the product $\star$
associated to the Hopf algebra $(\hgrd,m,\Delta)$:
\begin{align*}
T\star \id&=m\circ (\mu\otimes \id)\circ (\delta \otimes \id)\circ \Delta\\
&=m\circ(\mu\otimes \id \otimes \varepsilon' \otimes \id)\circ (\delta \otimes \delta)\circ \Delta\\
&=(\mu\otimes \varepsilon' \otimes \id)\circ m_{1,3,24}\circ (\delta \otimes \delta)\circ \Delta\\
&=(\mu\otimes \varepsilon' \otimes \id)\circ (\Delta \otimes \id)\circ \delta\\
&=(\mu \star \varepsilon'\otimes \id)\circ \delta\\
&=(\varepsilon \otimes \id)\circ \delta\\
&=\eta \circ \varepsilon,
\end{align*}
where $\eta:\Q\longrightarrow \hgrd$ send $1\in \Q$ on the empty graph (unit map). Consequently:
\[T=T\star \id \star S=(\eta \circ \varepsilon)\star S=S.\]

Let us now prove that for any graph $G\in \gr(\calD)$:
\[\mu(G)=P_{chr}^{(\calD)}(G)(-1).\]
As $P_{chr}^{(\calD)}$ is a Hopf algebra morphism from $(\hgrd,m,\Delta)$ to $(\Q[X],m,\Delta)$
and $\varepsilon'\circ P_{chr}^{(\calD)}=\varepsilon'$:
\[P_{chr}^{(\calD)}(G)(-1)=S\circ P_{chr}^{(\calD)}(G)(1)=\varepsilon'\circ S\circ P_{chr}^{(\calD)}(G)
=\varepsilon'\circ P_{chr}^{(\calD)}\circ S(G)=\varepsilon'\circ S(G)=\mu(G).\]
By Theorem \ref{theo23}:
\[\mu(G)= (-1)^{|G|} \sharp\{\mbox{acyclic orientations of $G$}\} .\qedhere\]
 \end{proof}

\section{Chromatic symmetric  functions}

\subsection{Reminders on $\QSym$}

The Hopf algebra $\QSym$ \cite{Aguiar2,Gelfand,Hazewinkel,MR3,Stanleyenu} 
has a basis $(M_u)$ indexed by compositions, that is to say finite sequences of positive integers.
Its product is given by quasi-shuffles. For example, if $a,b,c,d\in \N_{>0}$:
\begin{align*}
M_aM_{bcd}&=M_{abcd}+M_{bacd}+M_{bcad}+M_{bcda}+M_{(a+b)cd}+M_{b(a+c)d}+M_{ab(c+d)},\\
M_{ab}M_{cd}&=M_{abcd}+M_{acbd}+M_{acdb}+M_{cabd}+M_{cadb}+M_{cdab}\\
&+M_{(a+c)bd}+M_{(a+c)db}+M_{c(a+d)b}+M_{a(b+c)d}+M_{ac(b+d)}+M_{ca(b+d)}+M_{(a+b)(c+d)}.
\end{align*}
Its coproduct is given by deconcatenation: for any composition $w$,
\begin{align*}
\Delta(M_w)&=\sum_{uv=w}M_u\otimes M_v.
\end{align*}
For example, if $a,b,c\in \N_{>0}$:
\begin{align*}
\Delta(M_a)&=M_a\otimes 1+1\otimes M_a,\\
\Delta(M_{a,b})&=M_{a,b}\otimes 1+M_a\otimes M_b+1\otimes M_{a,b},\\
\Delta(M_{a,b,c})&=M_{a,b,c}\otimes 1+M_{a,b}\otimes M_c+M_a\otimes M_{b,c}+1\otimes M_{a,b,c}.
\end{align*}
For any composition $w$, we denote by $|w|$ the sum of its letters; this induces a connected graduation of $\QSym$.
There exists a second coproduct $\delta$, such that for any composition $w$ of length $n$:
\begin{align*}
\delta(M_w)&=\sum_{k=1}^n \sum_{w=w_1\ldots w_k}M_{|w_1|\ldots |w_k|}\otimes M_{w_1}\ldots M_{w_k}.
\end{align*}
For example, if $a,b,c\in \N_{>0}$:
\begin{align*}
\delta(M_a)&=M_a\otimes M_a,\\
\delta(M_{a,b})&=M_{a,b}\otimes M_aM_b+M_{a+b}\otimes M_{a,b},\\
\delta(M_{a,b,c})&=M_{a,b,c}\otimes M_aM_bM_c+M_{a+b,c}\otimes M_{a,b}M_c+M_{a,b+c}\otimes M_aM_{b,c}
+M_{a+b+c}\otimes M_{a,b,c}.
\end{align*}
The counit of this coproduct is denoted by $\varepsilon'$; for any composition $u$,
\begin{align*}
\varepsilon'(M_u)&=\begin{cases}
1\mbox{ if $u$ has only one letter},\\
0\mbox{ otherwise}.
\end{cases}
\end{align*}

Moreover, $\QSym$ admits a polynomial representation. Let $\X$ be a totally ordered alphabet -- that is to say a set with a total order.
For any  $u_1\ldots u_n\in \N_{>0}$, we consider the element:
\[\rep_\X(M_{u_1,\ldots,u_n})=\sum_{x_1<\ldots<x_n \mbox{\scriptsize{ in  }}\X} x_1^{u_1}\ldots x_n^{u_n} \in \Q[[\X]].\]
We define in this way an algebra morphism $\rep_\X:\QSym\longrightarrow \Q[[\X]]$. 
Moreover, for any $k\in \N$, the restriction of $\rep_\X$ to the $k$-th homogeneous component $\QSym_k$ of $\QSym$ is injective if,
 and only if, $|\X|\geq k$.\\

If $\X$ and $\Y$ are two totally ordered alphabets, $\X \sqcup \Y$ is also totally ordered: for all $x\in \X$, $y\in \Y$, $x\leq y$. 
We identify $\Q[[\X\sqcup \Y]]$ with $\Q[[\X]]\otimes \Q[[\Y]]$, via the continous morphism sending
$x\in \X$ to $x\otimes 1$ and $y\in \Y$ to $1\otimes y$.  Then:
\[\rep_{\X\sqcup \Y}=(\rep_\X\otimes \rep_\Y)\circ \Delta.\]
The cartesian product $\X\times \Y$ is totally ordered by the lexicographic order:
for any $x,x'\in \X$, $y,y' \in \Y$, $xy\leq x'y'$ if, and only if, $(x<x')$ or ($x=x'$ and $y\leq y'$).
We identify $\Q[[\X\times\Y]]$ with a subring of $\Q[[\X]]\otimes \Q[[\Y]]$ through the continous morphism
sending $(x,y)\in \X\times \Y$ to $x\otimes y$. Then:
\[\rep_{\X\times \Y}=(\rep_\X\otimes \rep_\Y)\circ \delta.\]
Let us prove the associativity of $\Delta$ and of $\delta$ and the cointeraction with the help of these polynomial representations.
We choose $X$, $Y$ and $Z$ three infinite totally alphabets. Firstly, observe that, as totally ordered alphabets:
\[(X\sqcup Y)\sqcup Z=X\sqcup (Y\sqcup Z).\]
Therefore:
\begin{align*}
(\rep_X \otimes \rep_Y\otimes \rep_Z)\circ (\Delta \otimes \id)\circ \Delta
&=\rep_{(X\sqcup Y)\sqcup Z}\\
&=\rep_{X\sqcup (Y\sqcup Z)}\\
&=(\rep_X \otimes \rep_Y\otimes \rep_Z)\circ (\id \otimes \Delta)\circ \Delta.
\end{align*}
As $\rep_X$, $\rep_Y$ and $\rep_Z$ are injective, $(\Delta \otimes \id)\circ \Delta=(\id \otimes \Delta)\circ \Delta$. 
Secondly, observe that, as totally ordered alphabets:
\[(X\times Y)\times Z=X\times (Y\times Z).\]
Therefore:
\begin{align*}
(\rep_X \otimes \rep_Y\otimes \rep_Z)\circ (\delta \otimes \id)\circ \delta
&=\rep_{(X\times Y)\times Z}\\
&=\rep_{X\times (Y\times Z)}\\
&=(\rep_X \otimes \rep_Y\otimes \rep_Z)\circ (\id \otimes \delta)\circ \delta.
\end{align*}
Hence, $(\delta \otimes \id)\circ \delta=(\id \otimes \delta)\circ \delta$. Finally, as totally ordered alphabets\footnote{
but $X\times (Y\sqcup Z)\neq (X\times Y)\sqcup (X\times Z)$ in general.}:
\[(X\sqcup Y)\times Z=(X\times Z)\sqcup (Y\times Z).\]
Therefore:
\begin{align*}
(\rep_X \otimes \rep_Y\otimes \rep_Z)\circ m_{1,3,24}\circ (\delta \otimes \delta)\circ \Delta
&=(\rep_X \otimes \rep_Z\otimes  \rep_Y\otimes \rep_Z)\circ(\delta \otimes \delta)\circ \Delta\\
&=\rep_{(X\times Z)\sqcup (Y\times Z)}\\
&=\rep_{(X\sqcup Y)\times Z}\\
&=(\rep_X \otimes \rep_Y\otimes \rep_Z)\circ (\Delta \otimes \id)\circ \delta.
\end{align*}
Hence, $m_{1,3,24}\circ (\delta \otimes \delta)\circ \Delta=(\Delta \otimes \id)\circ \delta$. We obtain:

\begin{prop}
With the coaction $\delta$, $(\QSym,m,\Delta)$  and $(\QSym,m,\delta)$ are in cointeraction.
\end{prop}

The Hopf algebra $\QSym$ contains the cocommutative Hopf subalgebra $\Sym$ of symmetric functions; 
this subalgebra is linearly generated by the elements:
\[M_{\{u_1,\ldots,u_k\}}=\sum_{\sigma\in \mathfrak{S}_k}M_{u_{\sigma(1)},\ldots u_{\sigma(k)}},\]
where $k\geqslant 1$ and $u_1,\ldots,u_k\in \N_{>0}$. 
Let us apply the results of \cite{FoissyEhrhart} to $\QSym$.

\begin{prop} \label{prop25}
For any $k\geq 0$, we denote by $H_k$ the $k$-th Hilbert polynomial:
\[H_k(X)=\frac{X(X-1)\ldots (X-k+1)}{k!}.\]
Let us consider the map:
\[H:\left\{\begin{array}{rlc}
\QSym&\longrightarrow&\Q[X]\\
M_{u_1\ldots u_k}&\longrightarrow&H_k.
\end{array}\right.\]
Then $H$ is the unique morphism from $\QSym$ to $\Q[X]$ compatible with $m$, $\Delta$ and $\delta$.
\end{prop}

\begin{proof} By \cite{FoissyEhrhart}, such a morphism exists and is unique. 
Let us prove that $H$ is indeed compatible with $m$, $\Delta$ and $\delta$.
For any finite totally ordered alphabet $\X$, of cardinality $k$, for any $a\in \QSym$, by definition of the polynomial representation of $\QSym$:
\begin{align*}
H(a)(k)&=\rep_\X(a)_{\mid \forall x\in \X, \: x=1}.
\end{align*}
If $a,b \in \QSym$, for any $k\geq 1$, if $\X$ is a totally ordered alphabet of cardinality $k$:
\begin{align*}
H(ab)(k)&=\rep_\X(ab)_{\mid \forall x\in \X, \: x=1}=\rep_\X(a)_{\mid \forall x\in \X, \: x=1}\rep_\X(b)_{\mid \forall x\in \X, \: x=1}=H(a)(k)H(b)(k).
\end{align*}
Hence, $H(ab)=H(a)H(b)$. 
If $a\in \QSym$, for any $k,l\geq 1$, choosing totally ordered alphabets $\X$ and $\Y$ of respective cardinality $k$ and $l$:
\begin{align*}
&\Delta\circ H(a)(k,l)	&&\delta\circ H(a)(k,l)	\\
&=H(a)(k+l)&&=H(a)(kl)\\
&=\rep_{\X\sqcup \Y}(a)_{\mid \forall x\in \X\sqcup \Y,\: x=1}&
&=\rep_{\X\times \Y}(a)_{\mid \forall x\in \X\sqcup \Y,\: x=1}\\
&=(\rep_\X \otimes \rep_\Y)\circ \Delta(a)_{\mid \forall x\in \X\sqcup \Y,\: x=1}&
&=(\rep_\X \otimes \rep_\Y)\circ \delta(a)_{\mid \forall x\in \X\sqcup \Y,\: x=1}\\
&=(H\otimes H)\circ \Delta(a)(k,l);&
&=(H\otimes H)\circ \delta(a)(k,l).
\end{align*}
Hence, $\Delta \circ H=(H\otimes H)\circ \Delta$ and  $\delta \circ H=(H\otimes H)\circ \delta$. \end{proof}

\subsection{Cointeraction and quasi-symmetric functions}

The following result is proved by Aguiar and Bergeron in \cite{Aguiar2}. 
It states that $\QSym$ is a terminal object in a suitable category of combinatorial Hopf algebras:

\begin{theo}\label{theo26} 
Let $(A,m,\Delta)$ be a graded, connected Hopf algebra, and $\alpha$ be a character on $A$. 
There exists a unique homogeneous Hopf algebra morphism
$\Phi_\alpha:(A,m,\Delta)\longrightarrow (\QSym,m,\Delta)$, such that $\alpha=\varepsilon'\circ \Phi_\alpha$. For any $a\in A$:
\[\Phi_\alpha(a)=\varepsilon(a)1+\sum_{k=1}^\infty \sum_{u_1,\ldots,u_k>0}\alpha^{\otimes k}\circ (\pi_{u_1}\otimes \ldots \otimes \pi_{u_k})\circ \Delta^{(k-1)}(a)
M_{u_1,\ldots,u_k},\]
where, for any $j\geq 1$, $\pi_j$ is the canonical projection on the $j$-th homogeneous component $A_j$ of $A$.
\end{theo}

\begin{theo}\label{theo27}
Let $(A,m,\Delta)$ and $(A,m,\delta)$ be cointeracting bialgebras, such that $(A,m,\Delta)$ is a graded connected Hopf algebra. 
We denote by $\varepsilon'$ the counit of the coalgebra $(A,\delta)$. 
\begin{enumerate}
\item There exists a morphism $\Phi_1:A\longrightarrow \QSym$ such that:
\begin{enumerate}
\item $\Phi_1:(A,m,\Delta)\longrightarrow (\QSym,m,\Delta)$ is a homogeneous morphism of Hopf algebras,
\item  $\Phi_1:(A,m,\delta)\longrightarrow (\QSym,m,\delta)$ is a morphism of bialgebras,
\end{enumerate}
if, and only if:
\begin{align*}
&\forall n\in \N,&\delta(A_n)\subseteq A_n \otimes A+\ker(\Phi_{\varepsilon'})\otimes A+A\otimes \ker(\Phi_{\varepsilon'}).
\end{align*}
Moreover, if this holds, then $\Phi_1=\Phi_{\varepsilon'}$, and the unique morphism $\phi_1:A\longrightarrow K[X]$
given by Theorem \ref{theo8} is $\Phi_1\circ H$.
\item  If:
\begin{align*}
&\forall n\in \N,&\delta(A_n)\subseteq A_n \otimes A+\ker(\Phi_{\varepsilon'})\otimes A,
\end{align*}
then for any character $\alpha$ on $A$, $\Phi_\alpha=\Phi_{\varepsilon'}\leftarrow \alpha$.
\end{enumerate}\end{theo}

\begin{proof}
1. \textit{Unicity}.  If $\Phi_1$ is such a morphism, then $\varepsilon'\circ \Phi_1=\varepsilon'$. By Theorem \ref{theo26},
$\Phi_1=\Phi_{\varepsilon'}$. From now, we put $\Phi_1=\Phi_{\varepsilon'}$.\\

\textit{Existence}, $\Longrightarrow$.  Let us assume that 
$\delta \circ \Phi_1=(\Phi_1\otimes \Phi_1)\circ \delta$. Let $x\in A_n$. Let us put
\[\delta(x)=\sum_{i=0}^\infty \sum_j x_{i,j}\otimes y_{i,j},\]
where $x_{i,j}\in A_i$ for any $(i,j)$.
As $\Phi_1$ is homogeneous, $\Phi_1(x)\in \QSym_n$. By definition of the coproduct $\delta$ of $\QSym$,
$\delta \circ \Phi_1(x)\in \QSym_n\otimes \QSym_n$. Hence:
\[(\Phi_1\otimes \Phi_1)\circ \delta(x)=\sum_{i=0}^\infty \sum_j \underbrace{\Phi_1(x_{i,j})}_{\in \QSym_i}\otimes
\Phi_1(y_{i,j})\in \QSym_n \otimes \QSym_n.\]
Hence:
\[i\neq n\Longrightarrow \sum_j \Phi_1(x_{i,j})\otimes \Phi_1(y_{i,j})=0.\]
So:
\[i\neq n\Longrightarrow \sum_j x_{i,j}\otimes y_{i,j} \in \ker(\Phi_1\otimes \Phi_1)=\ker(\Phi_1)\otimes A+A\otimes \ker(\Phi_1),\]
and finally $x\in A_n \otimes A+\ker(\Phi_1)\otimes A+A\otimes \ker(\Phi_1)$.\\

\textit{Existence}, $\Longleftarrow$. 
We shall use the polynomial representation of $\QSym$.
If $\X,\Y$ are totally ordered alphabets, as $\Phi_1$ is compatible with $\Delta$:
\begin{align*}
\rep_{\X\sqcup \Y}\circ\Phi_1&=(\rep_\X\otimes \rep_\Y)\circ \Delta\circ \Phi_1
=(\rep_\X\otimes \rep_\Y)\circ(\Phi_1\otimes \Phi_1)\circ \Delta.
\end{align*}

Let us prove that for any finite totally ordered alphabet $\X$, for any totally ordered alphabet $\Y$:
\[\rep_{\X\times \Y}\circ \Phi_1=(\rep_\X\otimes \rep_\Y)\circ (\Phi_1\otimes \Phi_1)\circ \delta.\]
We proceed by induction on $n=|\X|$. If $n=1$, we put $X=\{x\}$. 
Let $a\in A_k$, with $k\in \N$. By the hypothesis on $A$:
\[(\Phi_1\otimes \Phi_1)\circ \delta(a)\in  \QSym_k \otimes \QSym.\]
Therefore:
\begin{align*}
(\rep_\X\otimes \rep_\Y)(\Phi_1\otimes \Phi_1)\circ \delta(a)
&=x^k(\varepsilon' \otimes \rep_\Y)\circ (\Phi_1\otimes \Phi_1)\circ \delta(a)\\
&=x^k(\varepsilon' \otimes \rep_\Y\circ \Phi_1)\circ \delta(a)\\
&=x^k \underbrace{(\id_\Q \otimes \rep_\Y\circ \Phi_1)}_{=\rep_\Y\circ \Phi_1}
\circ \underbrace{(\varepsilon'\otimes \id)\circ \delta}_{=\id_A}(a)\\
&=x^k \rep_\Y\circ \Phi_1(a)\\
&=\rep_{\X\times \Y}\circ \Phi_1(a).
\end{align*}
Let us assume that the results holds for any totally ordered alphabet $\X'$ such that $|\X'|<|\X|$, with $|\X|\geq 2$. 
Let $x_n$ be the maximal element of $\X$. We put $\X'=\X\setminus \{x_n\}$ and $\X''=\{x_n\}$, such that $\X=\X'\sqcup \X''$. Then:
\[\X\times \Y=(\X'\sqcup \X'')\times \Y=(\X'\times \Y)\sqcup (\X''\times \Y),\]
so:
\begin{align*}
\rep_{\X\times \Y}\circ \Phi_1&=\rep_{(\X'\times \Y)\sqcup (\X''\times \Y)}\circ\Phi_1\\
&=(\rep_{\X'\times \Y}\otimes \rep_{\X''\times \Y}) \circ (\Phi_1\otimes \Phi_1)\circ \Delta\\
&=(\rep_{\X'}\otimes \rep_\Y\otimes \rep_{\X''}\otimes \rep_\Y)\circ 
(\Phi_1\otimes \Phi_1\otimes \Phi_1\otimes \Phi_1)\circ (\delta \otimes \delta)\circ \Delta\\
&=(\rep_{\X'}\otimes \rep_{\X''}\otimes \rep_\Y)\circ m_{1,3,24}\circ (\Phi_1\otimes \Phi_1\otimes \Phi_1\otimes \Phi_1)\circ (\delta \otimes \delta)\circ \Delta\\
&=(\rep_{\X'}\otimes \rep_{\X''}\otimes \rep_\Y)\circ(\Phi_1\otimes \Phi_1\otimes \Phi_1)\circ m_{1,3,24}\circ (\delta \otimes \delta)\circ \Delta\\
&=(\rep_{\X'}\otimes \rep_{\X''}\otimes \rep_\Y)\circ(\Phi_1\otimes \Phi_1\otimes \Phi_1)\circ (\Delta \otimes \id)\circ \delta\\
&=(\rep_{\X'\sqcup \X''}\otimes \rep_\Y)\circ (\Phi_1\otimes \Phi_1)\circ \delta\\
&=(\rep_\X\otimes \rep_\Y)\circ (\Phi_1\otimes \Phi_1)\circ \delta.
\end{align*}
Let $a\in A$. Let us choose a totally ordered alphabet $X$ of cardinality $n$ such that:
\[\delta(a)\in \bigoplus_{k,l\leq n} A_k\otimes A_l.\]
Then:
\[\rep_{\X \times \X}\circ \Phi_1(a)=(\rep_\X\otimes \rep_\X)\circ\delta\circ \Phi_1(a)=(\rep_\X\otimes \rep_\X)\circ(\Phi_1\otimes \Phi_1)\circ \delta(a).\]
By injectivity of $\rep_\X$  till degree $n$, as $|X|\geq n$, $\delta \circ \Phi_1(a)=(\Phi_1 \otimes \Phi_1)\circ \delta(a)$.\\

The morphism $\Phi_1\circ H:A\longrightarrow\QSym$ is compatible with both bialgebraic structures by composition.
By unicity in Theorem \ref{theo8}, it is equal to $\phi_1$. \\

2. Let $a\in A_n$. Then by hypothesis, $\Phi_1\leftarrow \alpha(a)=(\Phi_1\otimes \alpha)\circ \delta(a)\in \QSym_n$,
so $\Phi_1\leftarrow \alpha$ is a homogeneous Hopf algebra morphism. Moreover:
\[\varepsilon'\circ (\Phi_1\leftarrow \alpha)=(\varepsilon'\circ \Phi_1 \otimes \alpha) \circ \delta
=(\varepsilon' \otimes \alpha)\circ \delta=\alpha \circ (\varepsilon' \otimes \id)\circ \delta=\alpha,\]
so $(\Phi_1\leftarrow \alpha)=\Phi_\alpha$. \end{proof}

\subsection{Double morphisms from graphs to quasisymmetric functions}

\begin{notation}
For any graph $G\in \gr(\calD)$, for any $f\in \PC(G)$ and for any $i\in [\max(f)]$, we put:
\begin{align*}
\wt(f^{-1}(i))&=\sum_{\mbox{\scriptsize $G$ connected component of $G_{\mid f^{-1}(i)}$}}
\wt\left(\sum_{x\in V(G)} d(x)\right);\\ \\
M_f&=M_{\wt(f^{-1}(1))\ldots \wt(f^{-1}\max(f))}\in \QSym.
\end{align*}
In other words, $f^{-1}(i)$ is the sum of the weights of the connected components of the subgraph of $G$ 
which vertices are the vertices of $G$ colored by $i$.
\end{notation}

\begin{remark}
If $\wt:(\calD,+)\longrightarrow (\N_{>0},+)$ is a semigroup morphism, this simplify:
\[\wt(f^{-1}(i))=\wt\left(\sum_{x\in V(G),\: f(x)=i} d(x)\right). \]
\end{remark}

\begin{prop}
Let $\calD$ be a nonempty set. We define $F_{chr}^{(\calD)}:\hgrd\longrightarrow\QSym$ by:
\[F_{chr}^{(\calD)}:\left\{\begin{array}{rcl}
\hgrd&\longrightarrow&\QSym\\
G\in \gr(\calD)&\longrightarrow&\displaystyle \sum_{f\in \PVC(G)} M_{\wt(f^{-1}(1))\ldots \wt(f^{-1}(\max(f)))}.
\end{array}\right.\]
Then $F_{chr}^{(\calD)}$ is a Hopf algebra morphism, equal to $\phi_{\varepsilon'}$. 
\end{prop}

\begin{proof}
Let us apply Theorem \ref{theo26} in order to describe $\Phi_{\varepsilon'}$: for any nonempty
$G\in \gr(\calD)$,
\begin{align*}
\Phi_{\varepsilon'}(G)&=\sum_{k=1}^\infty \sum_{u_1,\ldots,u_k>0}\varepsilon'^{\otimes k}
\circ (\pi_{u_1}\otimes \ldots \otimes \pi_{u_k})\circ \Delta^{(k-1)}(G)M_{u_1,\ldots,u_k}\\
&=\sum_{k=1}^\infty \sum_{u_1,\ldots,u_k>0}
\sum_{V(G)=I_1\sqcup \ldots \sqcup I_k}
\varepsilon' \circ \pi_{u_1}(G_{\mid I_1})\ldots \varepsilon' \circ \pi_{u_k}(G_{\mid I_k})
M_{u_1,\ldots,u_k}\\
 &=\sum_{k=1}^\infty \sum_{V(G)=I_1\sqcup \ldots \sqcup I_k}
\varepsilon' (G_{\mid I_1})\otimes \ldots \otimes  \varepsilon' (G_{\mid I_k})
M_{\wt(G_{\mid I_1}),\ldots,\wt(G_{\mid I_k})}.
\end{align*}
Moreover, for any graph $H$, $\varepsilon'(H)=1$ if $H$ is totally disconnected and $0$ otherwise. Hence:
\begin{align*}
\Phi_{\varepsilon'}(G)&=\sum_{k=1}^\infty \sum_{\substack{V(G)=I_1\sqcup \ldots \sqcup I_k,\\ 
 \forall i\in [k], \: G_{\mid I_i}\mbox{\scriptsize totally disconected}}} M_{\wt(G_{\mid I_1}),\ldots,\wt(G_{\mid I_k})}\\
 &= \sum_{f\in \PVC(G)} M_{\wt(f^{-1}(1))\ldots \wt(f^{-1}(\max(f)))}\\
 &=F_{chr}^{(\calD)}(G).\qedhere 
\end{align*} \end{proof}

\begin{example}
Let $a,b,c\in \calD$.
\begin{align*}
F_{chr}^{(\calD)}(\grdun{$a$})&=M_{\wt(a)},\\
F_{chr}^{(\calD)}(\grddeux{$a$}{$b$})&=M_{\wt(a),\wt(b)}+M_{\wt(b),\wt(a)},\\
F_{chr}^{(\calD)}(\grdtroisun{$a$}{$b$}{$c$})&=M_{\wt(a),\wt(b),\wt(c)}+M_{\wt(a),\wt(c),\wt(b)}+M_{\wt(b),\wt(a),\wt(c)}\\
&+M_{\wt(b),\wt(c),\wt(a)}+M_{\wt(c),\wt(a),\wt(b)}+M_{\wt(c),\wt(b),\wt(a)},\\
F_{chr}^{(\calD)}(\grdtroisdeux{$a$}{$b$}{$c$})&=M_{\wt(a),\wt(b),\wt(c)}+M_{\wt(a),\wt(c),\wt(b)}+M_{\wt(b),\wt(a),\wt(c)}\\
&+M_{\wt(b),\wt(c),\wt(a)}+M_{\wt(c),\wt(a),\wt(b)}+M_{\wt(c),\wt(b),\wt(a)}\\
&+M_{\wt(a),\wt(b)+\wt(c)}+M_{\wt(b)+\wt(c),\wt(a)}.
\end{align*}
In the nondecorated case, this simplifies:
\begin{align*}
F_{chr}(\grun)&=M_1,&F_{chr}(\grtroisun)&=6M_{1,1,1},\\
F_{chr}(\grdeux)&=2M_{1,1},&F_{chr}(\grtroisdeux)&=6M_{1,1,1}+M_{1,2}+M_{2,1}.
\end{align*}
For any graph $G$, $F_{chr}(G)$ is the chromatic symmetric function of \cite{Stanley2},
when realized with the totally ordered alphabet $X=\{x_1<x_2<\ldots\}$. For example:
\begin{align*}
\rep_X\circ F_{chr}(\grun)&=\sum_{i=1}^\infty x_i,\\
\rep_X\circ F_{chr}(\grdeux)&=\sum_{\substack{i,j\geqslant 1\\ i\neq j}} x_i x_j,\\
\rep_X\circ F_{chr}(\grtroisun)&=\sum_{\substack{i,j,k\geqslant 1,\\ i\neq j,\\ i \neq k,\\ j\neq k}} x_i x_j x_k,\\
\rep_X\circ F_{chr}(\grtroisdeux)&=\sum_{\substack{i,j,k\geqslant 1,\\ i\neq j,\\ i \neq k}} x_i x_j x_k
=\sum_{\substack{i,j,k\geqslant 1,\\ i\neq j,\\ i \neq k,\\ j\neq k}} x_i x_j x_k
+\sum_{\substack{i,j\geqslant 1,\\ i\neq j}} x_i x_j^2.
\end{align*}
\end{example}

We now fix an abelian semigroup $(\calD,+)$ and a map $\wt:\calD\longrightarrow \N_{>0}$, inducing a graduation
on $\hgrd$. 

\begin{theo}\label{theo28}
There exists a morphism $\Phi_1:\hgrd\longrightarrow \QSym$ such that:
\begin{enumerate}
\item $\Phi_1:(\hgrd,m,\Delta)\longrightarrow (\QSym,m,\Delta)$ is a homogeneous morphism of Hopf algebras,
\item  $\Phi_1:(\hgrd,m,\delta)\longrightarrow (\QSym,m,\delta)$ is a morphism of bialgebras,
\end{enumerate}
if, and only if, $\wt:(\calD,+)\longrightarrow (\N_{>0},+)$ is a semigroup morphism. If so, $\Phi_1=F_{chr}^{(\calD)}$
(and therefore $\Phi_1$ is unique).
\end{theo}

\begin{proof}
$\Longrightarrow$. By Theorem \ref{theo27}, $\varepsilon' \circ \Phi_1=\varepsilon'$, 
so necessarily $\Phi_1=\Phi_{\varepsilon'}=F_{chr}^{(\calD)}$.
Let $a,b\in \calD$. Then:
\begin{align*}
\delta \circ \Phi_1(\grddeux{$a$}{$b$})&=(M_{\wt(a),\wt(b)}+M_{\wt(b),\wt(a)})\otimes M_{\wt(a)}M_{\wt(b)}\\
&+M_{\wt(a)+\wt(b)}\otimes (M_{\wt(a),\wt(b)}+M_{\wt(b),\wt(a)}),\\
=(\Phi_1\otimes \Phi_1)\circ \delta(\grddeux{$a$}{$b$})&=(\Phi_1\otimes \Phi_1)
(\grddeux{$a$}{$b$}\otimes \grdun{$a$}\grdun{$b$}+\grdunbis{$a+b$}\otimes \grddeux{$a$}{$b$})\\
&=(M_{\wt(a),\wt(b)}+M_{\wt(b),\wt(a)})\otimes M_{\wt(a)}M_{\wt(b)}\\
&+M_{\wt(a+b)}\otimes (M_{\wt(a),\wt(b)}+M_{\wt(b),\wt(a)}).
\end{align*}
Comparing, we obtain $\wt(a)+\wt(b)=\wt(a+b)$, so $\wt$ is a semigroup morphism.\\

$\Longleftarrow$. Let us assume that $\wt$ is a semigroup morphism. Let $G\in \gr(\calD)$ and $\sim\triangleleft G$. 
Then, obviously, $\wt(G\mid \sim)=\wt(G)$ and:
\begin{align*}
\wt(G/\sim)&=\sum_{c \in V(G/\sim)} \wt(d(c))\\
&=\sum_{c \in V(G/\sim)} \wt\left(\sum_{x\in c} d(x)\right)\\
&=\sum_{c\in V(G/\sim)}\sum_{x\in c} \wt(d(x))\\
&=\sum_{x\in V(G)} \wt(d(x))\\
&=\wt(G).
\end{align*}
Hence, for any $n\in \N$, $\delta((\hgrd)_n)\subseteq (\hgrd)_n\otimes (\hgrd)_n$. By Theorem \ref{theo27}-2,
$\Phi_{\varepsilon'}$ is a morphism for both bialgebraic structures. \end{proof}

\begin{example}
As a consequence, in the nondecorated case, $F_{chr}$ is not compatible with $\delta$. Indeed, for example:
\begin{align*}
&\delta \circ F_{chr}(\grdeux)&&(F_{chr}\otimes F_{chr})\circ \delta(\grdeux)\\
&=2\delta(M_{1,1})&&=(F_{chr}\otimes F_{chr})(\grdeux\otimes \grun \grun+\grun\otimes \grdeux)\\
&=2(M_{1,1}\otimes M_1M_1+\textcolor{red}{M_2} \otimes M_{1,1}),
&&=2(M_{1,1}\otimes M_1M_1+\textcolor{red}{M_1}\otimes M_{1,1}).
\end{align*}
On the other side, if $(\calD,+)=(\N_{>0},+)$ and $\wt=\id_{\N_{>0}}$, then $F_{chr}^{\calD)}$ is compatible with $\delta$:
\begin{align*}
&\delta \circ F_{chr}(\grddeux{$1$}{$1$})&&(F_{chr}\otimes F_{chr})\circ \delta(\grddeux{$1$}{$1$})\\
&=2\delta(M_{1,1})&&=(F_{chr}\otimes F_{chr})
(\grddeux{$1$}{$1$}\otimes \grdun{$1$} \grdun{$1$}+\grdun{$2$}\otimes \grddeux{$1$}{$1$})\\
&=2(M_{1,1}\otimes M_1M_1+M_2 \otimes M_{1,1}),&
&=2(M_{1,1}\otimes M_1M_1+M_2\otimes M_{1,1}).
\end{align*}\end{example}

\begin{prop}
The image of $F_{chr}^{(\calD)}$ is included in $\Sym$. It is equal to $\Sym$, if and only if,
there exists $a\in \calD$,  such that $\wt(a)=1$.  
\end{prop}

\begin{proof}
As $\hgrd$ is cocommutative, $F_{chr}^{(\calD)}(\hgrd)$ is a cocommutative Hopf subalgebra of $\QSym$, 
so is included in $\Sym$, greatest cocommutative subalgebra of $\QSym$. 

If $1\notin \wt(\calD)$, then there is no element of $\hgrd$ homogeneous of degree $1$.
As $F_{chr}^{(\calD)}$ is homogeneous, there is no element $x\in \hgrd$ such that $\Phi_1(x)=M_1$. \\

If $\wt(a)=1$, let us consider the complete graph $G_n$ with $n$ vertices, all decorated by $a$.
By definition of $F_{chr}^{(\calD)}$, $F_{chr}^{(\calD)}(G_n)=n! M_{1^n}$, so for any $n$, $M_{1^n}\in \Phi_1(\hgrd)$.
As these elements (which are the elementary symmetric functions) generate $\Sym$,
$\Phi_1(\hgrd)=\Sym$. \end{proof}

\subsection{Extension of $\phi_0$}

\begin{prop} \label{prop30}
Let $G$ be a graph and $f\in \PC(G)$.
 We define the equivalence $\sim_f$ in $V(G)$ as the unique one which classes are the connected components of 
the subsets $f^{-1}(x)$, $x\in [\max(f)]$. Then, the coloring $f$ induces a packed valid coloring $\overline{f}$ of $G/\sim_f$:
\begin{align*}
&\forall x\in V(G),&\overline{f}(\overline{x})&=f(x).
\end{align*}
\end{prop}

\begin{proof} We have to prove that $\overline{f}$ is a valid coloring of $G/\sim_f$. Let $\overline{x}$, $\overline{y}$ be two vertices of $G/\sim_f$,
related by an edge (this implies that they are different); we assume that $\overline{f}(\overline{x})=\overline{f}(\overline{y})$.
There exist $x',y'\in V(G)$, such that $x'\sim_f x$ and $y'\sim_f y$, and $x'$, $y'$ are related by an edge in $G$.
By definition of $\sim_f$, there exist vertices $x'=x_1,\ldots,x_k=x$, $y=y_1,\ldots,y_l=y'$ in $G$ such that $f(x_1)=\ldots=f(x_k)$,
$g(y_1)=\ldots=g(y_l)$, and for all $p$, $q$, $x_p$ and $x_{p+1}$, $y_q$ and $y_{q+1}$ are related by an edge in $G$.
Hence, there is a path in $G$ from $x$ to $y$, such that for any vertex $z$ on this path, $f(z)=f(x)=f(y)$: this implies that $x\sim_f y$,
so $\overline{x}=\overline{y}$. This is a contradiction, so $\overline{f}$ is valid. \end{proof}

\begin{prop} \label{prop31}
Let us consider the following map:
\[F_0^{(\calD)}:\left\{\begin{array}{rcl}
\hgrd&\longrightarrow&\QSym\\
G&\longrightarrow&\displaystyle \sum_{f\in \PC(G)} M_f.
\end{array}\right.\]
This is a Hopf algebra morphism, and $F_0^{(\calD)}\circ H=\phi_0^{(\calD)}$.
It is homogeneous if, and only if, $\wt:(\calD,+)\longrightarrow(\N_{>0},+)$ is a semigroup morphism.
 Moreover, in $E_{\hgrd\rightarrow \QSym}$:
\[F_{chr}^{(\calD)}=F_0^{(\calD)}\leftarrow \lambda_{chr}^{(\calD)}.\]
\end{prop}

\begin{proof} Let $G$ be graph. By Proposition \ref{prop30}, we have a map:
\[\theta:\left\{
\begin{array}{rcl}
\PC(G)&\longrightarrow&\displaystyle \bigsqcup_{\sim\triangleleft G} \PVC(G/\sim)\\
f&\longrightarrow&\overline{f}\in \PVC(G/\sim_f).
\end{array}\right.\]
$\theta$ is injective: if $\theta(f)=\theta(g)$, then $\sim_f=\sim_g$ and for any $x\in V(G)$,
\[f(x)=\overline{f}(\overline{x})=\overline{g}(\overline{x})=g(x).\]
Let us show that $\theta$ is surjective. Let $\overline{f}\in \PVC(G/\sim)$,
with $\sim \triangleleft G$. We define $f\in \PC(G)$ by $f(x)=\overline{f}(\overline{x})$ for any vertex $x$.
By definition of $f$, the equivalence classes of $\sim$ are included in sets $f^{-1}(i)$, and are connected, as $\sim\triangleleft G$,
so are included in equivalence classes of $\sim_f$: if $x\sim y$, then $x\sim_f y$. Let us assume that $x\sim_f y$. 
There exists a path $x=x_1,\ldots,x_k=y$ in $G$, such that $f(x_1)=\ldots=f(x_k)$. So $\overline{f}(\overline{x_1})=\ldots=\overline{f}(\overline{x_k})$.
As $\overline{f}$ is a valid coloring of $G/\sim$, there is no edge between $\overline{x_p}$ and $\overline{x_{p+1}}$ in $G/\sim$ for any $p$;
this implies that $\overline{x_p}=\overline{x_{p+1}}$ for any $p$, so $x=x_1\sim x_k=y$. Finally, $\sim=\sim_f$, so $\theta(f)=\overline{f}$.\\

Using the bijection $\theta$, we obtain:
\begin{align*}
F_0^{(\calD)}(G)&=\sum_{f\in \PC(G)} M_f\\
&=\sum_{\sim \triangleleft G} \sum_{\overline{f} \in \PVC(G/\sim)} M_{\overline{f}}\\
&=\sum_{\sim \triangleleft G} F_{chr}^{(\calD)}(G/\sim)\\
&=\sum_{\sim \triangleleft G} F_{chr}^{(\calD)}(G/\sim)\lambda_0^{(\calD)}(G\mid \sim)\\
&=\left(F_{chr}^{(\calD)}\leftarrow \lambda_0^{(\calD)}\right)(G).
\end{align*}
Therefore, $F_0^{(\calD)}=F_{chr}^{(\calD)}\leftarrow \lambda_0^{(\calD)}$, 
or equivalently  $F_{chr}^{(\calD)}=F_0^{(\calD)}\leftarrow \lambda_{chr}^{(\calD)}$.
As a consequence, $F_0^{(\calD)}$  is a Hopf algebra morphism, taking its values in $\Sym$.
Hence:
\begin{align*}
H\circ F_0^{(\calD)}&=H\circ(F_{chr}^{(\calD)}\leftarrow \lambda_0^{(\calD)})
=(H\circ F_{chr}^{(\calD)})\leftarrow \lambda_0^{(\calD)}=P_{chr}\leftarrow \lambda_0^{(\calD)}=\phi_0^{(\calD)}.
\end{align*}

Let us assume that $F_0^{(\calD)}$ is homogeneous. For any $a,b\in \calD$,
$F_0^{(\calD)}(\grddeux{$a$}{$b$})$ is homogeneous of degree $\wt(a)+\wt(b)$, so
$M_{\wt(a+b)}$ is homogeneous of degree $\wt(a)+\wt(b)$. Hence, $\wt(a+b)=\wt(a)+\wt(b)$, and $\wt$ is a semigroup morphism.
Conversely, if $G\in \gr(\calD)$, any term appearing in $F_0^{(\calD)}(G)$ is of degree
\[\sum_{x\in V(G)}\wt(d(x))=\wt(G),\]
so $F_0^{(\calD)}$ is homogeneous.  \end{proof}

\begin{example} Let $a,b,c\in \calD$. 
\begin{align*}
F_0^{(\calD)}(\grdun{$a$})&=M_{\wt(a)},\\
F_0^{(\calD)}(\grddeux{$a$}{$b$})&=M_{\wt(a),\wt(b)}+M_{\wt(b),\wt(a)}+M_{\wt(a+b)},\\
F_0^{(\calD)}(\grdtroisdeux{$a$}{$b$}{$c$})&=M_{\wt(a),\wt(b),\wt(c)}+M_{\wt(a),\wt(c),\wt(b)}+M_{\wt(b),\wt(a),\wt(c)}\\
&+M_{\wt(b),\wt(c),\wt(a)}+M_{\wt(c),\wt(a),\wt(b)}+M_{\wt(c),\wt(b),\wt(a)}\\
&+M_{\wt(a+b),\wt(c)}+M_{\wt(a+c),\wt(b)}+M_{\wt(b)+\wt(c),\wt(a)}\\
&+M_{\wt(c),\wt(a+b)}+M_{\wt(b),\wt(a+c)}+M_{\wt(a),\wt(b)+\wt(c)}+M_{\wt(a+b+c)},\\
F_0^{(\calD)}(\grdtroisun{$a$}{$b$}{$c$})&=M_{\wt(a),\wt(b),\wt(c)}+M_{\wt(a),\wt(c),\wt(b)}+M_{\wt(b),\wt(a),\wt(c)}\\
&+M_{\wt(b),\wt(c),\wt(a)}+M_{\wt(c),\wt(a),\wt(b)}+M_{\wt(c),\wt(b),\wt(a)}\\
&+M_{\wt(a+b),\wt(c)}+M_{\wt(a+c),\wt(b)}+M_{\wt(b+c),\wt(a)}\\
&+M_{\wt(c),\wt(a+b)}+M_{\wt(b),\wt(a+c)}+M_{\wt(a),\wt(b+c)}+M_{\wt(a+b+c)}.
\end{align*}
In the nondecorated case, this simplifies:
\begin{align*}
F_0(\grun)&=M_1,&F_0(\grtroisdeux)&=6M_{111}+4M_{11}+M_{12}+M_{21}+M_1,\\
F_0(\grdeux)&=2M_{11}+M_1,&F_0(\grtroisun)&=6M_{111}+6M_{11}+M_1.
\end{align*}
\end{example}

\section{Non-commutative versions}

\subsection{Non-commutative Hopf algebra of graphs}

\begin{defi}
\begin{enumerate}
\item An indexed graph is a graph $G$ such that $V(G)=[n]$, with $n\geq 0$. The set of indexed graphs is denoted by $\GR$.
\item Let $G=([n],E(G))$ be an indexed graph and let $I\subseteq [n]$. There exists a unique increasing bijection $f:I\longrightarrow [k]$,
where $k=\sharp I$. We denote by $G_{\mid I}$ the indexed graph defined by:
\[G_{\mid I}=([k],\{ \{f(x),f(y)\} \mid \{x,y\}\in E(G),\: x,y\in I\}).\] 
\item Let $G$ be an indexed graph and $\sim \triangleleft G$.
\begin{enumerate}
\item The graph $G|\sim$ is an indexed graph.
\item We order the elements of $V(G)/\sim$ by their minimal elements; using the unique increasing bijection from $V(G)/\sim$ to $[k]$,
$G/\sim$ becomes an indexed graph.
\end{enumerate}
\item Let $G=([k],E(G))$ and $H=([l],E(H))$ be indexed graphs. The indexed graph $GH$ is defined by:
\begin{align*}
V(GH)&=[k+l],\\
E(GH)&=E(G)\sqcup \{\{x+k,y+l\}\mid \{x,y\}\in E(H)\}.
\end{align*}\end{enumerate}\end{defi}

The Hopf algebra $(\hGR,m,\Delta)$ is, as its commutative version, introduced in \cite{Schmitt}:

\begin{theo} \begin{enumerate}
\item We denote by $\hGR$ the vector space generated by indexed graphs. We define a product $m$ and two coproducts $\Delta$ and $\delta$ 
on $\hGR$ in the following way: 
\begin{align*}
&\forall G,H\in \GR,& m(G\otimes H)&=GH,\\
&\forall G=([n],E(G))\in \GR,&\Delta(G)&=\sum_{I\subseteq [n]} G_{\mid I} \otimes G_{[n]\setminus I},\\
&\forall G\in \GR,& \delta(G)&=\sum_{\sim \triangleleft G} G/\sim \otimes G\mid \sim.
\end{align*}
Then $(\hGR,m,\Delta)$ is a graded cocommutative Hopf algebra, and $(\hGR,m,\delta)$ is a bialgebra.
\item Let $\varpi:\hGR\longrightarrow \hgr$ be the surjection sending an indexed graph to its isoclass. 
\begin{enumerate}
\item $\varpi:(\hGR,m,\Delta)\longrightarrow (\hgr,m,\Delta)$ is a surjective Hopf algebra morphism.
\item $\varpi:(\hGR,m,\delta)\longrightarrow (\hgr,m,\delta)$ is a surjective bialgebra morphism.
\item We put $\rho=(\id\otimes \varpi)\circ \delta:\hGR\longrightarrow \hGR\otimes \hgr$.
This defines a coaction of $(\hgr,m,\delta)$ on $\hGR$; moreover, $(\hGR,m,\Delta)$ is a Hopf algebra in the category of 
$(\hgr,m,\delta)$-comodules.
\end{enumerate}
\end{enumerate}\end{theo}

\begin{proof} 1. Similar to the proofs of Propositions \ref{prop2} and \ref{prop4}. \\

2. Points (a) and (b) are immediate; point (c) is proved in the same way as Theorem \ref{theo7}.\end{proof}

\begin{example}
\begin{align*}
\Delta(\grdun{$1$})&=\grdun{$1$}\otimes 1+1\otimes \grdun{$1$},\\
\Delta(\grddeux{$1$}{$2$})&=\grddeux{$1$}{$2$}\otimes 1+1\otimes \grddeux{$1$}{$2$}+\grdun{$1$}\otimes \grdun{$1$},\\
\Delta(\grdtroisun{$1$}{$3$}{$2$})&=\grdtroisun{$1$}{$3$}{$2$}\otimes 1+1\otimes \grdtroisun{$1$}{$3$}{$2$}
+3 \grdun{$1$}\otimes \grddeux{$1$}{$2$}+3\grddeux{$1$}{$2$}\otimes \grdun{$1$},\\
\Delta(\grdtroisdeux{$1$}{$3$}{$2$})&=\grdtroisdeux{$1$}{$3$}{$2$}\otimes 1+1\otimes \grdtroisdeux{$1$}{$3$}{$2$}
+2\grddeux{$1$}{$2$}\otimes \grdun{$1$}+\grdun{$1$}\grdun{$2$}\otimes \grdun{$1$}
+2\grdun{$1$}\otimes \grddeux{$1$}{$2$}+\grdun{$1$}\otimes \grdun{$1$}\grdun{$2$};\\ \\
\delta(\grdun{$1$})&=\grdun{$1$}\otimes \grdun{$1$},\\
\delta(\grddeux{$1$}{$2$})&=\grdun{$1$}\otimes\grddeux{$1$}{$2$}+\grddeux{$1$}{$2$}\otimes \grdun{$1$}\grdun{$2$},\\
\delta(\grdtroisun{$1$}{$3$}{$2$})&=\grdun{$1$}\otimes \grdtroisun{$1$}{$3$}{$2$}
+\grddeux{$1$}{$2$}\otimes(\grdun{$1$}\grddeux{$2$}{$3$}+\grdun{$2$}\grddeux{$1$}{$3$}+\grdun{$3$}\grddeux{$1$}{$2$})
+\grdtroisun{$1$}{$3$}{$2$}\otimes \grdun{$1$}\grdun{$2$}\grdun{$3$},\\
\delta(\grdtroisdeux{$1$}{$3$}{$2$})&=\grdun{$1$}\otimes \grdtroisdeux{$1$}{$3$}{$2$}
+\grddeux{$1$}{$2$}\otimes (\grdun{$2$}\grddeux{$1$}{$3$}+\grdun{$3$}\grddeux{$1$}{$2$})
+\grdtroisdeux{$1$}{$3$}{$2$}\otimes \grdun{$1$}\grdun{$2$}\grdun{$3$}.
\end{align*}\end{example}

\begin{remark}
 $(\hGR,m,\Delta)$ is not a bialgebra in the category of $(\hGR,m,\delta)$-comodules, as shown in the following example:
\begin{align*}
(\Delta\otimes \id)\circ \delta(\grdtroisdeux{$1$}{$3$}{$2$})&=\Delta(\grdun{$1$})\otimes \grdtroisdeux{$1$}{$3$}{$2$}
+(\grddeux{$1$}{$2$}\otimes 1+1\otimes \grddeux{$1$}{$2$})\otimes 
(\grdun{$2$}\grddeux{$1$}{$3$}+\grddeux{$1$}{$2$}\grdun{$3$})\\
&+\grdun{$1$}\otimes \grdun{$1$}\otimes (\textcolor{red}{\grdun{$2$}\grddeux{$1$}{$3$}}+\grddeux{$1$}{$2$}\grdun{$3$})
+\Delta(\grdtroisdeux{$1$}{$3$}{$2$})\otimes \grdun{$1$}\grdun{$2$}\grdun{$3$},\\
m_{1,3,24}\circ (\delta \otimes \delta)\circ \Delta(\grdtroisdeux{$1$}{$3$}{$2$})&=
\Delta(\grdun{$1$})\otimes \grdtroisdeux{$1$}{$3$}{$2$}
+(\grddeux{$1$}{$2$}\otimes 1+1\otimes \grddeux{$1$}{$2$})\otimes 
(\grdun{$2$}\grddeux{$1$}{$3$}+\grddeux{$1$}{$2$}\grdun{$3$})\\
&+\grdun{$1$}\otimes \grdun{$1$}\otimes (\textcolor{red}{\grdun{$1$}\grddeux{$2$}{$3$}}+\grddeux{$1$}{$2$}\grdun{$3$})
+\Delta(\grdtroisdeux{$1$}{$3$}{$2$})\otimes \grdun{$1$}\grdun{$2$}\grdun{$3$}.
\end{align*}\end{remark}

\subsection{Reminders on $\WQSym$}

Let us recall the construction of $\WQSym$ \cite{NovelliThibon2}. 
\begin{defi}
\begin{enumerate}
\item Let $w$ be a word with letters in $\N_{>0}$. We shall say that $w$ is \emph{packed} if:
\begin{align*}
&\forall j\in \N_{>0},&\mbox{$j$ appears in $w$}\Longrightarrow \mbox{$1,\ldots, j$ appear in $w$}.
\end{align*}
\item Let $w=x_1\ldots x_k$ a wordwith letters in $\N_{>0}$. 
There exists a unique increasing bijection $f$ from $\{x_1,\ldots,x_k\}$ to $[l]$,
with $l\geq 0$; the packed word $\pack(w)$ is $f(x_1)\ldots f(x_k)$.
\item $w=x_1\ldots x_k$ a word in $\N_{>0}$ and $I\subseteq \N_{>0}$. The word $w_{\mid I}$ is the word obtained by
taking the letters of $w$ which are in $I$.
\end{enumerate} \end{defi}

The Hopf algebra $\WQSym$ has the set of packed words for basis. If $w=w_1\ldots w_k$ and $w'=w'_1\ldots w'_l$ are packed words, then:
\[w\squplus w'=\sum_{\substack{w''=w''_1\ldots w''_{k+l},\\ \pack(w''_1\ldots w''_k)=w,\\ \pack(w''_{k+1}\ldots w''_{k+l}=w'}} w''.\]
For any packed word $w$:
\[\Delta(w)=\sum_{i=0}^{\max(w)} w_{\mid [i]} \otimes \pack(w_{[\max(w)]\setminus [i]}).\]
Then $(\WQSym,\squplus,\Delta)$ is a Hopf algebra. Moreover, $\WQSym$ has also a second coproduct $\delta$ defined on 
any packed word $w=w_1\ldots w_k$ by:
\[\delta(w)=\sum_{f,g} f(w_1)\ldots f(w_k)\otimes g(w_1)\ldots g(w_k),\]
where the sum runs over all pairs of maps $(f,g)$, where 
$f:[\max(w)]\longrightarrow[\max(f)]$ is an  increasing surjective map and $g:[\max(w)]\longrightarrow[\max(g)]$
is an increasing map such that for any $i\in [\max(f)]$, $g_{\mid f^{-1}(i)}$ is increasing.
However, $(\WQSym,\squplus,\Delta)$ is not a bialgebra in the category of right
$(\WQSym,\squplus,\delta)$-comodules, as shown in the following example. 
\begin{align*}
(\Delta \otimes \id)\circ \delta((132))&=\Delta((132))\otimes (1)\squplus (1) \squplus (1)\\
&+((121)\otimes1+1\otimes (121))\otimes ((112)+(121)+(132)+(123)+(213))\\
&+(11)\otimes (1)\otimes ( \textcolor{red}{(112)}+(121)+(132)+(123)+\textcolor{red}{(213)})\\
&+\Delta((122))\otimes (1)\squplus (11)+\Delta((111))\otimes (132),\\
\squplus_{1,3,24}\circ (\delta \otimes \delta)\circ \Delta((132))&=\Delta((132))\otimes (1)\squplus (1) \squplus (1)\\
&+((121)\otimes1+1\otimes (121))\otimes ((112)+(121)+(132)+(123)+(213))\\
&+(11)\otimes (1)\otimes ( \textcolor{red}{(121)}+(121)+(132)+(123)+\textcolor{red}{(231)})\\
&+\Delta((122))\otimes (1)\squplus (11)+\Delta((111))\otimes (132).
\end{align*}

This Hopf algebra admits a polynomial representation: 
we fix a infinite totally ordered alphabet $\X$; the set of words in $\X$ is denoted by $\X^*$. 
For any packed word $w$, we consider the noncommutative formal series:
\[\Rep_\X(w)=\sum_{w'\in \X^*,\: \pack(w')=w} w'\in \Q\langle\langle \X\rangle\rangle.\]
Then $\Rep_\X$ is an algebra morphism from $\WQSym$ to $\Q\langle\langle \X\rangle\rangle$.
For example:
\begin{align*}
\Rep_X(111)&=\sum_{x\in X}x^3,\\
\Rep_X(112)&=\sum_{x<y \:\mbox{\scriptsize in }X}x^2y,&\Rep_X(221)&=\sum_{x<y \:\mbox{\scriptsize in }X}y^2x,\\
\Rep_X(121)&=\sum_{x<y \:\mbox{\scriptsize in }X}xyx,&\Rep_X(212)&=\sum_{x<y \:\mbox{\scriptsize in }X}yxy,\\
\Rep_X(211)&=\sum_{x<y \:\mbox{\scriptsize in }X}yx^2,&\Rep_X(122)&=\sum_{x<y \:\mbox{\scriptsize in }X}xy^2,\\
\Rep_X(123)&=\sum_{x<y<z \:\mbox{\scriptsize in }X}xyz,&
\Rep_X(312)&=\sum_{x<y<z \:\mbox{\scriptsize in }X}zxy.
\end{align*}
If $\X$ is infinite, then $\Rep_\X$ is injective.
If $\X$ and $\Y$ are two totally ordered alphabets, 
we shall consider $\Q\langle \langle \X \rangle\rangle\otimes \Q\langle \langle \Y\rangle\rangle$ as a quotient of 
$\Q\langle \langle \X\sqcup \Y \rangle\rangle$, through the continuous map:
\[\left\{\begin{array}{rcl}
\Q\langle \langle \X\sqcup \Y \rangle\rangle&\longrightarrow&\Q\langle \langle \X \rangle\rangle\otimes \Q\langle \langle \Y\rangle\rangle\\
x\in X&\longrightarrow&x\otimes 1,\\
y\in Y&\longrightarrow &1\otimes y.
\end{array}\right.\]
We obtain:
\[\Rep_{\X\sqcup \Y}=(\Rep_\X\otimes \Rep_\Y)\circ \Delta.\]
We shall identify
$\Q\langle\langle \X\times \Y\rangle\rangle$ with a subalgebra of $\Q\langle\langle \X\rangle\rangle \otimes 
\Q\langle\langle \Y\rangle\rangle$, through the continuous map:
\[\left\{\begin{array}{rcl}
\Q\langle \langle \X\times \Y \rangle\rangle&\longrightarrow&\Q\langle \langle \X \rangle\rangle\otimes \Q\langle \langle \Y\rangle\rangle\\
(x,y)\in \X\times \Y&\longrightarrow&x\otimes y.
\end{array}\right.\]
We obtain:
\[\Rep_{\X\times \Y}=(\Rep_\X\otimes \Rep_\Y)\circ \delta.\]

\subsection{Non-commutative chromatic symmetric functions}

\begin{defi}
A set partition is a partition of a set $[n]$, with $n\geq 0$. The set of set partitions is denoted by $\SP$.
\end{defi}

\begin{theo} \label{theo35} \begin{enumerate}
\item For any packed word $w$ of length $n$ and of maximal $k$, we denote by $p(w)$ the set partition 
$\{w^{-1}(1),\ldots,w^{-1}(k)\}$.
For any set partition $\varpi\in \SP$, we put:
\[W_\varpi=\sum_{w\in \PW,\: p(w)=\varpi} w.\]
These elements are a basis of a cocommutative Hopf subalgebra of $\WQSym$, denoted by $\WSym$.
\item The following map is a Hopf algebra morphism from $(\hGR,m,\Delta)$ to $(\WQSym,\squplus,\Delta)$:
\[\bfF_{chr}:\left\{\begin{array}{rcl}
\hGR&\longrightarrow&\WQSym\\
G\in \GR&\longrightarrow&\displaystyle \sum_{f\in \PVC(G)} f(1)\ldots f(|G|).
\end{array}\right.\]
Its  image is $\WSym$. Moreover:
\begin{align*}
&\forall G \in \GR,&\bfF_{chr}(G)&=\sum_{\varpi\in \IP(G)} W_\varpi.
\end{align*}
\end{enumerate} \end{theo}

\begin{proof} 2. For any totally ordered alphabet $\X$, by definition of $\Rep_X$,
for any $G\in \GR$, with $n$ vertices:
\[\Rep_\X\circ \bfF_{chr}(G)=\sum_{f\in \VC(G,\X)}f(1)\ldots f(n).\]
Let us choose two infinite totally ordered alphabets $\X$ and $\Y$.
Let  $G$, $H\in \GR$, of respective degrees $m$ and $n$:
\begin{align*}
\Rep_\X(\bfF_{chr}(GH))&=\sum_{f\in \VC(GH,\X)}f(1)\ldots f(m+n)\\
&=\sum_{\substack{f'\in \VC(G,\X),\\ f''\in \VC(H,\X)}} f'(1)\ldots f'(m) f''(1)\ldots f''(n)\\
&=\Rep_\X\circ \bfF_{chr}(G)\Rep_\X\circ \bfF_{chr}(H)\\
&=\Rep_\X(\bfF_{chr}(G)\squplus \bfF_{chr}(H)).
\end{align*}
As $\Rep_\X$ is injective, $\bfF_{chr}(GH)=\bfF_{chr}(G)\squplus\bfF_{chr}(H)$, so $\bfF_{chr}$ is an algebra morphism.\\

Let $G\in \GR$, of degree $n$. 
\begin{align*}
(\Rep_\X\otimes \Rep_\Y)\circ \Delta \circ \bfF_{chr}(G)&=\Rep_{\X\sqcup \Y}\circ \bfF_{chr}(G)\\
&=\sum_{f\in \VC(FG,\X\sqcup \Y)}f(1)\ldots f(n)\\
&=\sum_{V(G)=I\sqcup J} \sum_{\substack{f'\in \VC(F_{\mid I},\X),\\f''\in \VC(F_{\mid J},\Y)}}
f'(1)\ldots f'(|I|) f''(1)\ldots f''(|J|)\\
&=\sum_{V(G)=I\sqcup J} \Rep_\X \circ \bfF_{chr}(G_{\mid I})\otimes \Rep_\Y \circ \bfF_{chr}(G_{\mid J})\\
&=(\Rep_\X\otimes \Rep_\Y)\circ (\bfF_{chr}\otimes \bfF_{chr})\circ \Delta(G).
\end{align*}
As $\Rep_\X$ and $\Rep_\Y$ are injective, $\Delta \circ \bfF_{chr}=(\bfF_{chr}\otimes \bfF_{chr})\circ \Delta$.\\

1. So $\WSym$ is a Hopf subalgebra of $\WQSym$, isomorphic to a quotient of $\hGR$, so is cocommutative. \end{proof}

\begin{remark}
\begin{enumerate}
\item The Hopf algebra $\WSym$, known as the Hopf algebra of word symmetric functions, is described and used 
in \cite{Bergeron,Bultel,Hivert}. Here is a description of its product and coproduct, with immediate notations: 
\begin{itemize}
\item For any set partitions $\varpi$, $\varpi'$ of respective degree $m$ and $n$:
\[W_\varpi W_{\varpi'}=\sum_{\substack{\varpi''\in \SP,\: \deg(\varpi'')=k+l,\\ 
\pack(\varpi''_{\mid [k]})=\varpi,\\\pack(\varpi''_{\mid [k+l]\setminus [k]})=\varpi'}}W_{\varpi''}.\]
\item For any set partition $\varpi=\{P_1,\ldots,P_k\}$:
\[\Delta(P_w)=\sum_{I\subseteq [k]} W_{\pack(\{I_p\mid p\in I\})}\otimes W_{\pack(\{I_p\mid p\notin I\})}.\]
\end{itemize}
For example:
\begin{align*}
W_{\{\{1,2\}\}} W_{\{\{1\}\}}&=W_{\{\{1,2\},\{3\}\}}+W_{\{\{1,2,3\}\}},\\
W_{\{\{1\},\{2\}\}}W_{\{\{1\}\}}&=W_{\{\{1\},\{2\},\{3\}\}}+W_{\{\{1,3\},\{2\}\}}+W_{\{\{1\},\{2,3\}\}},\\
\Delta(W_{\{\{1,3\},\{2\},\{4\}\}})&=W_{\{\{1,3\},\{2\},\{4\}\}}\otimes 1+W_{\{\{1,3\},\{2\}\}}\otimes W_{\{\{1\}\}}+W_{\{\{1,2\},\{3\}\}}\otimes W_{\{\{1\}\}}\\
&+W_{\{\{1\},\{2\}\}}\otimes W_{\{\{1,2\}\}}+W_{\{\{1,2\}\}}\otimes W_{\{\{1\},\{2\}\}}+W_{\{\{1\}\}}\otimes W_{\{\{1,2\},\{3\}\}}\\
&+W_{\{\{1\}\}}\otimes W_{\{\{1,3\},\{2\}\}}+1\otimes W_{\{\{1,3\},\{2\},\{4\}\}}.
\end{align*}

\item The map $\bfF_{chr}$ is not a bialgebra morphism from $(\hGR,m,\delta)$ to $(\WQSym,\squplus,\delta)$. For example:
\begin{align*}
(\bfF_{chr}\otimes \bfF_{chr})\circ \delta(\grddeux{$1$}{$2$})
&=\textcolor{red}{(1)}\otimes ((12)+(21))+((12)+(21))\otimes ((11)+(12)+(21)),\\
\delta\circ \bfF_{chr}(\grddeux{$1$}{$2$})
&=\textcolor{red}{(11)}\otimes ((12)+(21))+((12)+(21))\otimes ((11)+(12)+(21)).
\end{align*}
\end{enumerate}
\end{remark}

\subsection{Non-commutative version of $F_0$}

We shall use the notations of Proposition \ref{prop30}. If $G$ be an indexed graph and $f\in \PC(G)$, then
$G/\sim_f$ is an indexed graph; we denote its cardinality by $k$. We put:
\[w_f=\overline{f}(1)\ldots \overline{f}(k).\]

\begin{prop} \label{prop36}
Let us consider the following map:
\[\bfF_0:\left\{\begin{array}{rcl}
\hGR&\longrightarrow&\WSym\\
G&\longrightarrow&\displaystyle \sum_{f\in \PC(G)} w_f.
\end{array}\right.\]
This is a Hopf algebra morphism. Moreover, in $E_{\hGR\rightarrow \WSym}$:
\[\bfF_{chr}=\bfF_0\leftarrow \lambda_{chr}.\]
\end{prop}

\begin{proof} This is proved in the same way as Proposition \ref{prop31}. \end{proof}

\begin{example}
\begin{align*}
\bfF_0(\grdun{$1$})&=(1),\\
\bfF_0(\grddeux{$1$}{$2$})&=(12)+(21)+(1),\\
\bfF_0(\grdtroisun{$1$}{$3$}{$2$})&=(123)+(132)+(213)+(231)+(312)+(321)+3(12)+3(21)+(1),\\
\bfF_0(\grdtroisdeux{$1$}{$3$}{$2$})&=(123)+(132)+(213)+(231)+(312)+(321)+(122)+(211)+2(12)+2(21)+(1).
\end{align*}\end{example}

\subsection{From non-commutative to commutative}

As $\Q[[\X]]$ is a quotient of $\Q\langle\langle X\rangle\rangle$, this polynomial representations $\Rep$ of $\WQSym$ 
and $\rep$ of $\QSym$ induce a surjective Hopf algebra morphism:
\[\pi:\left\{\begin{array}{rcl}
\WQSym&\longrightarrow&\QSym\\
w&\longrightarrow&\displaystyle M_{|w^{-1}(1)|,\ldots, |w^{-1}(\max(w))|}.
\end{array}\right.\]

\begin{prop}
$\pi\circ \bfF_0=F_0\circ \varpi$ and $\pi\circ \bfF_{chr}=F_{chr}\circ \varpi$.
\end{prop}

\begin{proof} Immediate. \end{proof}

We obtain commutative diagrams of Hopf algebra morphisms:
\begin{align*}
&\xymatrix{\WQSym\ar@{->>}^{\pi}[r]&\QSym\ar@{->>}^H[r]&\Q[X]\\
\hGR\ar^{\bfF_{chr}}[u]\ar@{->>}[r]_\varpi&\hgr\ar^{F_{chr}}[u]\ar@{->>}_{P_{chr}}[ru]&}
&\xymatrix{\WQSym\ar@{->>}^{\pi}[r]&\QSym\ar@{->>}^H[r]&\Q[X]\\
\hGR\ar^{\bfF_0}[u]\ar@{->>}[r]_\varpi&\hgr\ar^{F_0}[u]\ar@{->>}_{\phi_0}[ru]&}
\end{align*}

\bibliographystyle{amsplain}
\bibliography{biblio}

\end{document}

%% file: graphes.tex
\newcommand{\grun}{\begin{tikzpicture}[line cap=round,line join=round,>=triangle 45,x=0.5cm,y=0.5cm]
\clip(-0.2,-0.1) rectangle (0.2,0.2);
\begin{scriptsize}
\draw [fill=black] (0.,0.) circle (1pt);
\end{scriptsize}
\end{tikzpicture}}

\newcommand{\grdeux}{\begin{tikzpicture}[line cap=round,line join=round,>=triangle 45,x=0.5cm,y=0.5cm]
\clip(-.2,-.1) rectangle (0.2,0.7);
\draw [line width=.5pt] (0.,0.5)-- (0.,0.);
\begin{scriptsize}
\draw [fill=black] (0.,0.) circle (1pt);
\draw [fill=black] (0.,0.5) circle (1pt);
\end{scriptsize}
\end{tikzpicture}}

\newcommand{\grtroisun}{\begin{tikzpicture}[line cap=round,line join=round,>=triangle 45,x=0.5cm,y=0.5cm]
\clip(-0.5,-0.1) rectangle (0.5,0.7);
\draw [line width=0.5pt] (0.,0.)-- (-0.3,0.5);
\draw [line width=0.5pt] (0.,0.)-- (0.3,0.5);
\draw [line width=0.5pt] (-0.3,0.5)-- (0.3,0.5);
\begin{scriptsize}
\draw [fill=black] (-0.3,0.5) circle (1pt);
\draw [fill=black] (0.,0.) circle (1pt);
\draw [fill=black] (0.3,0.5) circle (1pt);
\end{scriptsize}
\end{tikzpicture}}
\newcommand{\grtroisdeux}{\begin{tikzpicture}[line cap=round,line join=round,>=triangle 45,x=0.5cm,y=0.5cm]
\clip(-0.5,-0.1) rectangle (0.5,0.7);
\draw [line width=0.5pt] (0.,0.)-- (-0.3,0.5);
\draw [line width=0.5pt] (0.,0.)-- (0.3,0.5);
\begin{scriptsize}
\draw [fill=black] (-0.3,0.5) circle (1pt);
\draw [fill=black] (0.,0.) circle (1pt);
\draw [fill=black] (0.3,0.5) circle (1pt);
\end{scriptsize}
\end{tikzpicture}}

\newcommand{\grquatreun}{\begin{tikzpicture}[line cap=round,line join=round,>=triangle 45,x=0.5cm,y=0.5cm]
\clip(-0.2,-0.1) rectangle (0.7,0.7);
\begin{scriptsize}
\draw [fill=black] (0.,0.) circle (1pt);
\draw [fill=black] (0.,0.5) circle (1pt);
\draw [fill=black] (0.5,0.5) circle (1pt);
\draw [line width=.5pt] (0.,0.)-- (0.5,0.);
\draw [line width=0.5pt] (0.,0.)-- (0.,0.5);
\draw [line width=0.5pt] (0.5,0.)-- (0.5,0.5);
\draw [line width=0.5pt] (0.0,0.)-- (0.5,0.5);
\draw [line width=0.5pt] (0.0,0.5)-- (0.5,0.5);
\draw [line width=0.5pt] (0.0,0.5)-- (0.5,0.);
\draw [fill=black] (0.5,0.) circle (1pt);
\end{scriptsize}
\end{tikzpicture}}
\newcommand{\grquatredeux}{\begin{tikzpicture}[line cap=round,line join=round,>=triangle 45,x=0.5cm,y=0.5cm]
\clip(-0.2,-0.1) rectangle (0.7,0.7);
\begin{scriptsize}
\draw [fill=black] (0.,0.) circle (1pt);
\draw [fill=black] (0.,0.5) circle (1pt);
\draw [fill=black] (0.5,0.5) circle (1pt);
\draw [line width=.5pt] (0.,0.)-- (0.5,0.);
\draw [line width=0.5pt] (0.,0.)-- (0.,0.5);
\draw [line width=0.5pt] (0.5,0.)-- (0.5,0.5);
\draw [line width=0.5pt] (0.,0.)-- (0.5,0.5);
\draw [line width=0.5pt] (0.,0.5)-- (0.5,0.5);
\draw [fill=black] (0.5,0.) circle (1pt);
\end{scriptsize}
\end{tikzpicture}}
\newcommand{\grquatretrois}{\begin{tikzpicture}[line cap=round,line join=round,>=triangle 45,x=0.5cm,y=0.5cm]
\clip(-0.2,-0.1) rectangle (0.7,0.7);
\begin{scriptsize}
\draw [fill=black] (0.,0.) circle (1pt);
\draw [fill=black] (0.,0.5) circle (1pt);
\draw [fill=black] (0.5,0.5) circle (1pt);
\draw [line width=.5pt] (0.,0.)-- (0.5,0.);
\draw [line width=0.5pt] (0.,0.)-- (0.,0.5);
\draw [line width=0.5pt] (0.5,0.)-- (0.5,0.5);
\draw [line width=0.5pt] (0.,0.)-- (0.5,0.5);
\draw [fill=black] (0.5,0.) circle (1pt);
\end{scriptsize}
\end{tikzpicture}}
\newcommand{\grquatrequatre}{\begin{tikzpicture}[line cap=round,line join=round,>=triangle 45,x=0.5cm,y=0.5cm]
\clip(-0.2,-0.1) rectangle (0.7,0.7);
\begin{scriptsize}
\draw [fill=black] (0.,0.) circle (1pt);
\draw [fill=black] (0.,0.5) circle (1pt);
\draw [fill=black] (0.5,0.5) circle (1pt);
\draw [line width=.5pt] (0.,0.)-- (0.5,0.);
\draw [line width=0.5pt] (0.,0.)-- (0.,0.5);
\draw [line width=0.5pt] (0.5,0.)-- (0.5,0.5);
\draw [line width=0.5pt] (0.,0.5)-- (0.5,0.5);
\draw [fill=black] (0.5,0.) circle (1pt);
\end{scriptsize}
\end{tikzpicture}}
\newcommand{\grquatrecinq}{\begin{tikzpicture}[line cap=round,line join=round,>=triangle 45,x=0.5cm,y=0.5cm]
\clip(-0.2,-0.1) rectangle (0.7,0.7);
\begin{scriptsize}
\draw [fill=black] (0.,0.) circle (1pt);
\draw [fill=black] (0.,0.5) circle (1pt);
\draw [fill=black] (0.5,0.5) circle (1pt);
\draw [line width=.5pt] (0.,0.)-- (0.5,0.);
\draw [line width=0.5pt] (0.,0.)-- (0.,0.5);
\draw [line width=0.5pt] (0.,0.)-- (0.5,0.5);
\draw [fill=black] (0.5,0.) circle (1pt);
\end{scriptsize}
\end{tikzpicture}}
\newcommand{\grquatresix}{\begin{tikzpicture}[line cap=round,line join=round,>=triangle 45,x=0.5cm,y=0.5cm]
\clip(-0.2,-0.1) rectangle (0.7,0.7);
\begin{scriptsize}
\draw [fill=black] (0.,0.) circle (1pt);
\draw [fill=black] (0.,0.5) circle (1pt);
\draw [fill=black] (0.5,0.5) circle (1pt);
\draw [line width=.5pt] (0.,0.)-- (0.5,0.);
\draw [line width=0.5pt] (0.,0.)-- (0.,0.5);
\draw [line width=0.5pt] (0.5,0.)-- (0.5,0.5);
\draw [fill=black] (0.5,0.) circle (1pt);
\end{scriptsize}
\end{tikzpicture}}

\newcommand{\grdun}[1]{\begin{tikzpicture}[line cap=round,line join=round,>=triangle 45,x=0.5cm,y=0.5cm]
\clip(-.2,-.1) rectangle (0.5,0.5);
\begin{scriptsize}
\draw [fill=black] (0.,0.) circle (1pt);
\end{scriptsize}
\draw(0.3,0.1) node {\tiny #1};
\end{tikzpicture}}

\newcommand{\grdunbis}[1]{\begin{tikzpicture}[line cap=round,line join=round,>=triangle 45,x=0.5cm,y=0.5cm]
\clip(-.2,-.1) rectangle (1.5,.5);
\begin{scriptsize}
\draw [fill=black] (0.,0.) circle (1pt);
\end{scriptsize}
\draw(0.8,0.1) node {\tiny #1};
\end{tikzpicture}}

\newcommand{\grdunter}[1]{\begin{tikzpicture}[line cap=round,line join=round,>=triangle 45,x=0.5cm,y=0.5cm]
\clip(-.2,-.1) rectangle (2.5,.5);
\begin{scriptsize}
\draw [fill=black] (0.,0.) circle (1pt);
\end{scriptsize}
\draw(1.3,0.1) node {\tiny #1};
\end{tikzpicture}}

\newcommand{\grddeux}[2]{\begin{tikzpicture}[line cap=round,line join=round,>=triangle 45,x=0.5cm,y=0.5cm]
\clip(-.2,-.1) rectangle (0.5,1.);
\draw [line width=.5pt] (0.,0.5)-- (0.,0.);
\begin{scriptsize}
\draw [fill=black] (0.,0.) circle (1pt);
\draw [fill=black] (0.,0.5) circle (1pt);
\end{scriptsize}
\draw(0.3,0.1) node {\tiny #1};
\draw(0.3,0.6) node {\tiny #2};
\end{tikzpicture}}

\newcommand{\grddeuxbis}[2]{\begin{tikzpicture}[line cap=round,line join=round,>=triangle 45,x=0.5cm,y=0.5cm]
\clip(-.2,-.1) rectangle (1.5,1.);
\draw [line width=.5pt] (0.,0.5)-- (0.,0.);
\begin{scriptsize}
\draw [fill=black] (0.,0.) circle (1pt);
\draw [fill=black] (0.,0.5) circle (1pt);
\end{scriptsize}
\draw(0.8,0.1) node {\tiny #1};
\draw(0.3,0.6) node {\tiny #2};
\end{tikzpicture}}

\newcommand{\grdtroisun}[3]{\begin{tikzpicture}[line cap=round,line join=round,>=triangle 45,x=0.5cm,y=0.5cm]
\clip(-0.5,-0.1) rectangle (0.5,1.2);
\draw [line width=0.5pt] (0.,0.)-- (-0.3,0.5);
\draw [line width=0.5pt] (0.,0.)-- (0.3,0.5);
\draw [line width=0.5pt] (-0.3,0.5)-- (0.3,0.5);
\begin{scriptsize}
\draw [fill=black] (-0.3,0.5) circle (1pt);
\draw [fill=black] (0.,0.) circle (1pt);
\draw [fill=black] (0.3,0.5) circle (1pt);
\end{scriptsize}
\draw(0.35,0.1) node {\tiny #1};
\draw(0.3,0.8) node {\tiny #2};
\draw(-0.3,0.8) node {\tiny #3};
\end{tikzpicture}}
\newcommand{\grdtroisdeux}[3]{\begin{tikzpicture}[line cap=round,line join=round,>=triangle 45,x=0.5cm,y=0.5cm]
\clip(-0.5,-0.1) rectangle (0.5,1.2);
\draw [line width=0.5pt] (0.,0.)-- (-0.3,0.5);
\draw [line width=0.5pt] (0.,0.)-- (0.3,0.5);
\begin{scriptsize}
\draw [fill=black] (-0.3,0.5) circle (1pt);
\draw [fill=black] (0.,0.) circle (1pt);
\draw [fill=black] (0.3,0.5) circle (1pt);
\end{scriptsize}
\draw(0.35,0.1) node {\tiny #1};
\draw(0.3,0.8) node {\tiny #2};
\draw(-0.3,0.8) node {\tiny #3};
\end{tikzpicture}}